\documentclass[10pt]{amsart}
\usepackage[margin=1in,letterpaper,portrait]{geometry}
\usepackage{graphicx,xypic}
\usepackage{latexsym,amssymb,verbatim}
 \usepackage{multirow}

\theoremstyle{plain}
   \newtheorem{theorem}{Theorem}[section]
   \newtheorem{proposition}[theorem]{Proposition}
   \newtheorem{lemma}[theorem]{Lemma}
   \newtheorem{corollary}[theorem]{Corollary}
   \newtheorem{conjecture}[theorem]{Conjecture}
   
   \newtheorem*{theorem*}{Theorem}
   \newtheorem*{reinterp}{Proposition~\ref{CSP-proposition}$'$} 
\theoremstyle{definition}
   \newtheorem{definition}[theorem]{Definition}
   \newtheorem{example}[theorem]{Example}
   
   \newtheorem{question}[theorem]{Question}
   \newtheorem{remark}[theorem]{Remark}

\numberwithin{equation}{section}

\newcommand\qbin[3]{\left[\begin{matrix} #1 \\ #2 \end{matrix} \right]_{#3}}
\newcommand\ugpark[1]{{#1}[\FF_{q^m}^n]}

\newcommand\Symm{\mathfrak{S}}

\newcommand\xx{{\mathbf{x}}}

\newcommand\rank{\operatorname{rank}}
\newcommand\one{\mathbf{1}}
\newcommand\Sym{\operatorname{Sym}}

\newcommand\Ind{\operatorname{Ind}}
\newcommand\Ann{\operatorname{Ann}}
\newcommand\Cat{\operatorname{Cat}}
\newcommand\Hilb{\operatorname{Hilb}}
\newcommand\Frac{\operatorname{Frac}}

\newcommand\Aut{\operatorname{Aut}}

\newcommand\Mat{\operatorname{Mat}}
\newcommand\spanof{\operatorname{span}}
\newcommand\triv{\operatorname{triv}}

\newcommand\mm{\mathfrak{m}}
\newcommand\nn{\mathfrak{n}}
\newcommand\gr{\mathfrak{gr}}

\newcommand\CC{{\mathbb{C}}}
\newcommand\ZZ{{\mathbb{Z}}}
\newcommand\NN{{\mathbb{N}}}

\newcommand\RR{{\mathbb{R}}}
\newcommand\FF{{\mathbb{F}}}
\newcommand\Gr{{\mathbb{G}}}
\newcommand\Proj{{\mathbb{P}}}

\begin{document}

\title[Invariants of $GL_n(\FF_q)$ in polynomials mod
Frobenius powers] 
{Invariants of $GL_n(\FF_q)$ in polynomials mod
Frobenius powers}

\author{J. Lewis}
\author{V. Reiner}
\author{D. Stanton}
\email{(jblewis,reiner,stanton)@math.umn.edu}
\address{School of Mathematics\\
University of Minnesota\\
Minneapolis, MN 55455, USA}

\thanks{Work partially supported by NSF grants DMS-1148634 and
DMS-1001933.}

\keywords{finite general linear group, Frobenius power, 
Dickson invariants, cofixed, fixed quotient, coinvariant, divided power invariants,
reflection group, Catalan, parking, cyclic sieving}

\begin{abstract}
Conjectures are given for Hilbert series related to polynomial invariants of finite general linear groups, one for invariants mod Frobenius powers of the irrelevant ideal, one for cofixed spaces of polynomials.   
\end{abstract}

\date{\today}

\maketitle


\section{Introduction}
\label{intro-section}

This paper proposes two related conjectures 
in the invariant theory of $GL_n(\FF_q)$,
motivated by the following celebrated result of L.E. Dickson \cite{Dickson}; see also \cite[Thm. 8.1.1]{Benson}, \cite[Thm. 8.1.5]{Smith}.
\begin{theorem*}
When $G:=GL_n(\FF_q)$
acts via invertible linear substitutions of variables
on the polynomial algebra $S=\FF_q[x_1,\ldots,x_n]$,
the $G$-invariants form a polynomial subalgebra
$
S^G=\FF_q[D_{n,0}, D_{n,1}, \ldots, D_{n,n-1}].
$
\end{theorem*}

\noindent
Here the {\it Dickson polynomials}
$D_{n,i}$ are the coefficients in the expansion
$
\prod_{\ell(\xx)} (t + \ell(\xx)) = \sum_{i=0}^{n} D_{n,i} t^{q^i} 
$
where the product runs over all $\FF_q$-linear forms
$\ell(\xx)$ in the variables $x_1,\ldots,x_n$
.
In particular, $D_{n,i}$ is 
homogeneous of degree $q^n-q^i$, so that Dickson's Theorem 
implies this {\it Hilbert series} formula:
\begin{equation}
\label{Dickson-Hilbert-series}
\Hilb(S^G,t):= \sum_{d \geq 0} \dim_{\FF_q} (S^G)_d \,\, t^d
=\prod_{i=0}^{n-1} \frac{1}{1-t^{q^n-q^i}}.
\end{equation}

Our main conjecture gives the Hilbert series for the $G$-invariants
in the quotient ring 
$
Q:=S/\mm^{[q^m]}
$ 
by an iterated {\it Frobenius power}
$
\mm^{[q^m]}:=(x_1^{q^m},\ldots,x_n^{q^m})
$
of the \emph{irrelevant ideal} $\mm=(x_1,\ldots,x_n)$.
The ideal $\mm^{[q^m]}$ is $G$-stable, and hence the $G$-action
on $S$ descends to an action on the quotient $Q$.

\begin{conjecture}
\label{mod-powers-conjecture}
The $G$-fixed subalgebra $Q^G$ has Hilbert series
$
\Hilb( (S/\mm^{[q^m]})^G, t)=C_{n,m}(t)
$ 
where
\begin{equation}
\label{Catalan-with-m-definition}
C_{n,m}(t):= \sum_{k=0}^{\min(n,m)} t^{(n-k)(q^m-q^k)} \qbin{m}{k}{q,t}.\\
\end{equation}
\end{conjecture}
The {\it $(q,t)$-binomial} appearing in \eqref{Catalan-with-m-definition} is a polynomial in $t$,  introduced and studied in \cite{StantonR}, defined by
\begin{equation}
\label{(q,t)-binomial-definition}
\qbin{n}{k}{q,t}
:= \frac{ \Hilb( S^{P_k}, t ) }
        {\Hilb( S^{G}, t) }
= \prod_{i=0}^{k-1} \frac{1-t^{q^n-q^i}}{1-t^{q^k-q^i}}.
\end{equation}
Here $P_k$ is a {\it maximal parabolic subgroup} of $G$ stabilizing
$\FF_q^k \subset \FF_q^n$, so  $G/P_k$
is the {\it Grassmannian} of $k$-planes.

It will be shown in Section~\ref{conjecture-implication-section}
that Conjecture~\ref{mod-powers-conjecture}
implies the following conjecture on the {\it $G$-cofixed space}
(also known as 
the {\it maximal $G$-invariant quotient} or the
{\it $G$-coinvariant space}\footnote{Warning:  
this last terminology is often used for a {\it different} object, 
the quotient ring $S/(D_{n,0},\ldots,D_{n,n-1})$, so we avoid it.}) of $S$.
This is defined to be the quotient $\FF_q$-vector space 
$
S_G:=S/N
$
where $N$ is the 
$\FF_q$-linear span of all polynomials  $g(f)-f$
with $f$ in $S$ and $g$ in $G$.

\begin{conjecture}
\label{cofixed-conjecture}
The $G$-cofixed space of $S=\FF_q[x_1,\ldots,x_n]$
has Hilbert series
$$
\Hilb(S_G,t) 
=\sum_{k = 0}^n t^{n(q^k - 1)}\prod_{i=0}^{k-1} \frac{1}{1 - t^{q^k - q^i}}.
%
$$
\end{conjecture}
\noindent
(Here and elsewhere we interpret empty products as $1$, as in the $k = 0$ summand above.)

\begin{example}
\label{small-n-example}

When $n=0$, 
Conjectures~\ref{mod-powers-conjecture} and \ref{cofixed-conjecture}
have little to say, 
since $S=\FF_q$ has no variables and $G=GL_0(\FF_q)$ is the trivial group.
When $n=1$, both conjectures are easily verified as follows.
The group 
$
G=GL_1(\FF_q)=\FF_q^\times
$
is cyclic of order $q-1$.
A cyclic generator $g$ for $G$
scales the monomials in $S=\FF_q[x]$ via
$
g(x^k)=(\zeta x)^k=\zeta^k x
$
where $\zeta$ is a $(q-1)$st root of unity in $\FF_q$;
$g$ similarly scales the monomial basis elements
$\{1,\overline{x},\overline{x}^2,\ldots,\overline{x}^{q^m-1}\}$ of the quotient ring $Q=S/\mm^{[q^m]}$.  
Hence $\overline{x}^k$ is $G$-invariant in $Q$ 
if and only if $q-1$ divides $k$, so that $Q^G$ has 
$\FF_q$-basis $\{1,\overline{x}^{q-1},\overline{x}^{2(q-1)},\ldots,\overline{x}^{q^m-q},\overline{x}^{q^m-1}\}$.
Therefore
$$
\Hilb(Q^G,t) = \left( 1+t^{q-1}+t^{2(q-1)} + \cdots + t^{q^m-q} \right)
+ t^{q^m-1}
=t^{0}\qbin{m}{1}{q,t} + t^{q^m-1} \qbin{m}{0}{q,t}
= C_{1,m}(t).
$$
For the same reason,  the image of $x^k$ survives as an $\FF_q$-basis
element in the $G$-cofixed quotient $S_G$ if and only if $q-1$ divides $k$.
Hence $S_G$ has $\FF_q$-basis given by the images of
$\{1, x^{q-1}, x^{2(q-1)},\ldots\}$,
so that 
\[
\Hilb(S_G,t) = 1+t^{q-1}+t^{2(q-1)} + \cdots =\frac{1}{1-t^{q-1}} = 1 + \frac{t^{q-1}}{1-t^{q-1}}.
\]
\end{example}

\subsection{The parabolic generalization}
In fact, we will work with generalizations of  
Conjectures~\ref{mod-powers-conjecture} and
\ref{cofixed-conjecture} to a {\it parabolic subgroup}
$P_\alpha$ of $G$ specified by a {\it composition}
$\alpha=(\alpha_1,\ldots,\alpha_\ell)$ of $n$, so that
$
|\alpha|:=\alpha_1+\cdots+\alpha_\ell =n,
$
and $\alpha_i>0$ without loss of generality.  This $P_\alpha$ is the subgroup 
of block upper-triangular invertible matrices
$$
g=\left[
\begin{matrix}
g_1 & *  & \cdots & * \\
0  & g_2 & \cdots & * \\ 
\vdots & \vdots& \ddots& \vdots\\
0     & 0 & \cdots&g_\ell
\end{matrix}
\right]
$$
with diagonal blocks $g_1, \ldots, g_\ell$ of sizes $\alpha_1 \times \alpha_1,  \ldots, \alpha_\ell \times \alpha_\ell$.
A generalization of Dickson's Theorem by Kuhn and Mitchell \cite{KuhnMitchell} 
(related to results of Mui \cite{Mui}, 
and rediscovered by Hewett \cite{Hewett})
asserts that $S^{P_\alpha}$ is again a polynomial algebra,
having Hilbert series given by the following expression,
where we denote partial sums of $\alpha$ by $A_i:=\alpha_1+\cdots+\alpha_i$:
\begin{equation}
\label{parabolic-invariants-hilbert-series}
\Hilb(S^{P_\alpha},t) 
= \prod_{i=1}^\ell \prod_{j=0}^{\alpha_i-1}
\frac{1}
{1 - t^{q^{A_i}-q^{A_{i-1}+j}}}.
 \end{equation}
This leads to a polynomial in $t$ called the  {\it $(q,t)$-multinomial}, also studied in \cite{StantonR}:
\begin{equation}
\label{(q,t)-multinomial-definition}
\qbin{n}{\alpha}{q,t}
:= \frac{ \Hilb(S^{P_\alpha},t) }{ \Hilb(S^G,t) }
= \frac{ 
\displaystyle
\prod_{j=0}^{n-1} (1-t^{q^n-q^j})
}
{
\displaystyle 
\prod_{i=1}^\ell \prod_{j=0}^{\alpha_i-1}
(1 - t^{q^{A_i}-q^{A_{i-1}+j}})
}.
\end{equation}
To state the parabolic versions of the conjectures, we consider
{\it weak compositions}  $\beta=(\beta_1,\ldots,\beta_\ell)$
with $\beta_i \in \ZZ_{\geq 0}$, of a fixed length $\ell$,
and partially order them {\it componentwise},
that is, $\beta \leq \alpha$ if $\beta_i \leq \alpha_i$ for $i=1,2,\ldots,\ell$.
In this situation, let $B_i:=\beta_1+\beta_2+\cdots+\beta_i$.

\vskip.1in
\noindent
{\bf Parabolic Conjecture~\ref{mod-powers-conjecture}.}
{\it
For  $m \geq 0$ and for $\alpha$ a composition of $n$,
the $P_\alpha$-fixed subalgebra $Q^{P_\alpha}$ of the quotient ring
$Q=S/\mm^{[q^m]}$  has Hilbert series $\Hilb(Q^{P_\alpha},t)=C_{\alpha,m}(t)$, where
\begin{equation}
\label{parabolic-Catalan-with-m-definition}
C_{\alpha,m}(t):=
 \sum_{\substack{\beta: 
\beta \leq \alpha\\ |\beta| \leq m}}
 t^{e(m,\alpha,\beta)}  \qbin{m}{\beta,m-|\beta|}{q,t}
\qquad \text{ with } \qquad
e(m,\alpha,\beta):=\sum_{i=1}^\ell  (\alpha_i-\beta_i) (q^m-q^{B_i}).
\end{equation}
}

\noindent
The $\ell=1$ case of Parabolic Conjecture~\ref{mod-powers-conjecture} is
Conjecture~\ref{mod-powers-conjecture}.  
Parabolic Conjecture~\ref{mod-powers-conjecture} also implies the following conjecture, whose $\ell=1$ case is Conjecture~\ref{cofixed-conjecture}.  
\vskip.1in
\noindent
{\bf Parabolic Conjecture~\ref{cofixed-conjecture}.}
{\it
For a composition $\alpha$ of $n$, the $P_\alpha$-cofixed space $S_{P_\alpha}$ 
of $S$ has Hilbert series 
\begin{equation}
\label{parabolic-cofixed-hilbert-series}
\Hilb(S_{P_\alpha},t) =
  \sum_{\substack{\beta:
\beta \leq \alpha}}
  t^{\sum_{i=1}^\ell \alpha_i(q^{B_i}-1)}
 \prod_{i=1}^\ell
 \prod_{j=0}^{\beta_i-1}
 \frac{1}
 {1 - t^{q^{B_i}-q^{B_{i-1}+j}}} .
\end{equation}
}

\subsection{Structure of the paper}
The rest of the paper explains
the relation between Parabolic 
Conjectures~\ref{mod-powers-conjecture} and~\ref{cofixed-conjecture}, 
along with context and evidence for both,
including relations to known results.

Section~\ref{consistency-with-Dickson-section}
explains why Parabolic Conjecture~\ref{mod-powers-conjecture} implies
the Hilbert series \eqref{parabolic-invariants-hilbert-series}
in the limit as $m \rightarrow \infty$, with proof delayed until Appendix~\ref{proof-of-equality-up-to-q^m-prop-section}.
 
Section~\ref{conjecture-implication-section} shows that
Parabolic Conjecture~\ref{mod-powers-conjecture} implies 
Parabolic Conjecture~\ref{cofixed-conjecture}.  It then shows
the reverse implication in the case $n=2$.  Appendix~\ref{n=2-section} proves both 
via direct arguments for $n=2$.

Section~\ref{small-m-section} checks 
Parabolic Conjecture~\ref{mod-powers-conjecture}
for $m=0,1$.

Section~\ref{module-over-invariants-section} explains why
the $P_\alpha$-cofixed space $S_{P_\alpha}$ is a 
finitely generated module of rank one
over the $P_\alpha$-fixed algebra $S^{P_\alpha}$, and why this is consistent with
the form of Parabolic Conjecture~\ref{cofixed-conjecture}.

Section~\ref{CSP-section} concerns some of our original combinatorial
motivation, comparing two $G$-representations:
\begin{itemize}
\item on the graded quotient $Q=S/\mm^{[q^m]}$, versus
\item permuting the points of $(\FF_{q^m})^n$.  
\end{itemize}
These two representations are {\it not} isomorphic; however, we 
will show that they have the same composition factors, that is,
they are {\it Brauer-isomorphic}.  
After extending scalars from $\FF_q G$ to $\FF_{q^m} G$-modules, this 
Brauer-isomorphism holds even taking into account a 
commuting group action $G \times C$, 
where the cyclic group $C=\FF_{q^m}^\times$
is the multiplicative group of $\FF_{q^m}$.
Consistent with this, Parabolic Conjecture~\ref{mod-powers-conjecture}
has a strange implication:  the two representations
have $G$-fixed spaces and $P_\alpha$-fixed spaces which
are {\it isomorphic} $C$-representations.
This assertion is equivalent to the fact that
evaluating $C_{\alpha,m}(t)$ when $t$ is a $(q^m-1)$st root of unity 
exhibits a {\it cyclic sieving phenomenon} 
in the sense of \cite{StantonWhiteR}. 

Section~\ref{questions-and-remark-section} collects some further questions and remarks.

\section*{Acknowledgements}
The authors thank  A. Broer, M. Crabb, N. Kuhn, A. Shepler, L. Smith, and P. Webb for 
valuable suggestions and references, as well as D. Stamate for
computations discussed in
Example~\ref{Stanley-decomposition-example-for-n=3}
regarding Question~\ref{Stanley-decomposition-conjecture}.

\section{Conjecture~\ref{mod-powers-conjecture} implies
\eqref{parabolic-invariants-hilbert-series}}
\label{consistency-with-Dickson-section}

The following proposition is delicate to verify, but serves two purposes,  
explained after its statement.

\begin{proposition}
\label{equality-up-to-q^m-prop}
For any $m \geq 0$ and any composition $\alpha$ of $n$, the power
series 
\[
\Hilb(S^{P_\alpha},t)
=
\prod_{i=1}^\ell \prod_{j=0}^{\alpha_i-1}
\frac{1}{1 - t^{q^{A_i}-q^{A_{i-1}+j}}}
\]
is congruent in $\ZZ[[t]]/(t^{q^m})$ to the polynomial
\[
C_{\alpha,m}(t) 
=
\displaystyle \sum_{\substack{\beta: 
\beta \leq \alpha\\ |\beta| \leq m}}
 t^{e(m,\alpha,\beta)}
 \qbin{m}{\beta,m-|\beta|}{q,t}
\qquad \text{ where } \qquad
e(m,\alpha,\beta) =\sum_{i=1}^\ell  (\alpha_i-\beta_i) (q^m-q^{B_i}).
\]
\end{proposition}

The first purpose of Proposition~\ref{equality-up-to-q^m-prop} is to give
evidence for Parabolic Conjecture~\ref{mod-powers-conjecture},
since it is implied by the conjecture:  the ideal $\mm^{[q^m]}=(x_1^{q^m},\ldots,x_n^{q^m})$ only contains elements
of degree $q^m$ and above, so the $G$-equivariant
quotient map 
$
S \twoheadrightarrow Q=S/\mm^{[q^m]}
$
restricts to $\FF_q$-vector space isomorphisms 
\begin{eqnarray}
S_d &\cong& Q_d
\label{S-and-Q-isomorphic-in-low-degrees}\\
S^{P_\alpha}_d &\cong& Q^{P_\alpha}_d  \notag
\end{eqnarray}
for $0 \leq d \leq q^m-1$.
Consequently one has  
\begin{equation}
\label{invariants-agree-in-low-degree}
\Hilb(S^{P_\alpha},t) \equiv \Hilb\left( Q^{P_\alpha},t \right) \bmod (t^{q^m}).
\end{equation}
In particular, Proposition~\ref{equality-up-to-q^m-prop}
shows why Parabolic Conjecture~\ref{mod-powers-conjecture}
 gives \eqref{parabolic-invariants-hilbert-series}
in the limit as $m \rightarrow \infty$.

Secondly, the precise form of Proposition~\ref{equality-up-to-q^m-prop} will be used in 
the proof of Corollary~\ref{n=2-equivalence-corollary}, asserting the
equivalence of Parabolic Conjectures~\ref{mod-powers-conjecture} 
and \ref{cofixed-conjecture} for $n=2$.

The proof of Proposition~\ref{equality-up-to-q^m-prop} is rather technical, so it is delayed until Appendix~\ref{proof-of-equality-up-to-q^m-prop-section}.

\section{Conjecture~\ref{mod-powers-conjecture} implies
Conjecture~\ref{cofixed-conjecture}}
\label{conjecture-implication-section}

The desired implication 
will come from an examination of the quotient ring
$$ 
Q:=S/\mm^{[q^m]} = \FF_q[x_1,\ldots,x_n]/(x_1^{q^m},\ldots,x_n^{q^m})
$$
as a {\it monomial complete intersection}, and hence a {\it Gorenstein ring}.
Note that $Q$ has monomial basis 
\begin{equation}
\label{monomial-basis-for-quotient}
\{\xx^a:=x_1^{a_1} \cdots x_n^{a_n}\}_{0 \leq a_i \leq q^m-1}
\end{equation}
and that its homogeneous component $Q_{d_0}$ of top degree 
\begin{equation}
\label{top-degree}
d_0:=n(q^m-1)
\end{equation}
is $1$-dimensional, spanned over $\FF_q$ by the image of the monomial 
$$
\xx^{a_0}:=(x_1 \cdots x_n)^{q^m-1}.
$$
Furthermore, the $\FF_q$-bilinear pairing
\begin{equation}
\label{Gorenstein-pairing}
\begin{aligned}
Q_i \otimes Q_j &\longrightarrow Q_{d_0} =\FF_q \cdot \xx^{a_0}  \cong \FF_q \\
(f_1, f_2) &\longmapsto  f_1 \cdot f_2 
\end{aligned}
\end{equation}
is {\it non-degenerate} (or {\it perfect}):  for monomials $\xx^a, \xx^b$ in \eqref{monomial-basis-for-quotient}
of degrees $i,j$ with $i+j=d_0$,  one has
$$
(\xx^a, \xx^b) = 
\begin{cases} 
\xx^{a_0} & \text{ if } a + b = a_0, \\
0 & \text{ otherwise.}
\end{cases}
$$

\begin{proposition}
\label{socle-is-invariant-prop}
The monomial $\xx^{a_0}=(x_1 \cdots x_n)^{q^m-1}$ has $G$-invariant
image in the quotient $Q=S/\mm^{[q^m]}$, and hence its span $Q_{d_0}$ carries
the trivial $G$-representation.
\end{proposition}
\begin{proof}[Proof 1.]
As $G$ acts on $S$ and on $Q$ preserving degree, it induces a $1$-dimensional
$G$-representation on $Q_{d_0}$.  Thus $Q_{d_0}$  must carry one of the 
linear characters of $G=GL_n(\FF_q)$,
that is, $\det^j$ for some $j$ in $\{0,1,\ldots,q-2\}$.
We claim that in fact $j=0$, since the element $g$ in $G$ that scales
the variable $x_1$ by  a primitive $(q-1)$st root of unity  
$\gamma$ in $\FF_q^\times$ and fixes all other variables $x_i$ with $i \geq 2$
will have $\det(g)=\gamma$ and has $g(\xx^{a_0})=\gamma^{q^m-1}\xx^{a_0}=\xx^{a_0}$.
\end{proof}
\begin{proof}[Proof 2.]
Note $G=GL_n(\FF_q)$ is generated by 
all {\it permutations} of coordinates, all {\it scalings} 
of coordinates,
and any {\it transvection}, such as the element $u$ sending 
$x_1 \mapsto x_1+ x_2$ and fixing $x_i$ for $i \neq 1$.
So it suffices to check that the image of 
$\xx^{a_0}=(x_1 \cdots x_n)^{q^m-1}$ in $Q$ is
invariant under permutations (obvious), 
invariant under scalings of a coordinate (easily checked as in Proof 1),
and invariant under the transvection $u$:
$$
u(\xx^{a_0})=(x_1+x_2)^{q^m-1} (x_2 \cdots x_n)^{q^m-1}
=(x_1^{q^m-1} + x_2 h) (x_2 \cdots x_n)^{q^m-1}
\equiv \xx^{a_0} \bmod \mm^{[q^m]},
$$
where $h$ is a polynomial whose exact form is unimportant.
\end{proof}

Note that Proposition~\ref{socle-is-invariant-prop} is an expected consequence
of Conjecture~\ref{mod-powers-conjecture}, due to the
following observation.

\begin{proposition}
For any composition $\alpha$ of $n$, the polynomial $C_{\alpha,m}(t)$ is monic 
of degree $d_0=n(q^m-1)$.
\end{proposition}
\begin{proof}
Letting $\deg_t(-)$ denote degree in $t$,
the product formula 
\eqref{(q,t)-multinomial-definition} for the $(q,t)$-multinomial shows
that 
\begin{equation}
\label{first-degree-contribution}
\begin{aligned}
\deg_t \qbin{m}{\beta,m-|\beta|}{q,t}
&= \sum_{j=0}^{|\beta|} (q^m-q^j) -
     \sum_{i=1}^\ell \sum_{j=0}^{\beta_i-1} (q^{B_i}-q^{B_{i-1}+j}) \\
&= |\beta| q^m - \sum_{j=0}^{|\beta|} q^j
- \sum_{i=1}^\ell \beta_i q^{B_i}
+ \sum_{i=1}^\ell \sum_{j=0}^{\beta_i-1} q^{B_{i-1}+j} \\ 
&= |\beta|q^m-\sum_{i=1}^{\ell} \beta_i q^{B_i},
\end{aligned}
\end{equation}
while the exponent on the monomial $t^{e(m,\alpha,\beta)}$ can be rewritten
\begin{equation}
\label{second-degree-contribution}
e(m,\alpha,\beta)
=\sum_{i=1}^\ell (\alpha_i-\beta_i)(q^m-q^{B_i})
=n q^m - |\beta| q^m -\sum_{i=1}^{\ell} \alpha_i q^{B_i}
+\sum_{i=1}^{\ell} \beta_i q^{B_i}.
\end{equation}
Therefore the summand of $C_{\alpha,m}(t)$ indexed by $\beta$ has degree equal to the 
sum of \eqref{first-degree-contribution} and \eqref{second-degree-contribution}, namely 
$$
n q^m - \sum_{i=1}^{\ell} \alpha_i q^{B_i} \geq 
n q^m - \sum_{i=1}^\ell \alpha_i = n q^m-n = n(q^m-1) = d_0.
$$
Equality occurs in this inequality
if and only $B_i=0$ for all $i$, 
so the $t$-degree is maximized uniquely by the $\beta=0$ summand, which is
the single monomial $t^{n(q^m-1)}=t^{d_0}$.
\end{proof}

Proposition~\ref{socle-is-invariant-prop} shows that the
nondegenerate pairing \eqref{Gorenstein-pairing} is {\it $G$-invariant}: for any $g$ in $G$, one has
$$
(g(f_1), g(f_2)) = g(f_1) g(f_2) = g(f_1 f_2) = f_1 f_2 = (f_1,f_2).
$$
Thus one has an isomorphism of $G$-representations
$
Q_i \cong Q_j^* \text{ in complementary degrees }i+j=d_0.
$
Here the notation $U^*$ denotes the representation {\it contragredient} or {\it dual} to the $G$-representation $U$ on its dual space, in which for any
functional $\varphi$ in $U^*$, group element $g$ in $G$ 
and vector $u$ in $U$, one has
$g(\varphi)(u)=\varphi(g^{-1}(u))$.
Cofixed spaces are dual to fixed spaces, as the following well-known proposition shows.

\begin{proposition}
\label{fixed-cofixed-duality}
For any group $G$ and any $G$-representation $U$ over a field $k$,
one has a $k$-vector space isomorphism
$
(U_G)^* \cong (U^*)^G,
$
in which $U_G$ is the cofixed space for $G$ acting on $U$, and $(U^*)^G$ is the
subspace of $G$-fixed functionals in $U^*$.
\end{proposition}
\begin{proof}
Recall that $U_G:=U/N$ where $N$ is the $k$-span of $\{g(u)-u\}_{u \in U, g \in G}$.
Thus, by the universal property of quotients, $(U_G)^*$ 
is  the subspace 
of functionals $\varphi$ in $U^*$ vanishing on restriction to $N$.
This is equivalent to $0=\varphi(g(u)-u)=\varphi(g(u))-\varphi(u)$ 
for all $u$ in $U$ and $g$ in $G$, 
that is, to $\varphi$ lying in $(U^*)^G$.
\end{proof}

\begin{corollary}
\label{complementary-degree-corollary}
For complementary degrees $i+j=d_0$ in $Q=S/\mm^{[q^m]}$,
one has an $\FF_q$-vector space duality of fixed and cofixed spaces
$(Q_i^{P_\alpha})^* \cong (Q_j)_{P_\alpha}$, and hence equality of their dimensions. 
Therefore one has
\begin{align}
\Hilb(Q_{P_\alpha},t)&=t^{d_0} \Hilb(Q^{P_\alpha},t^{-1}),
\label{reciprocal-matches-quotient-cofixed-space}\\
\Hilb(S_{P_\alpha},t) &\equiv t^{d_0} \Hilb(Q^{P_\alpha},t^{-1}) \quad \bmod{(t^{q^m})},
\label{reciprocal-matches-polynomial-cofixed-space}
\qquad \text{ and }\\
\Hilb(S_{P_\alpha},t) &=\lim_{m \rightarrow \infty} t^{d_0} \Hilb(Q^{P_\alpha},t^{-1})
\label{limit-of-reciprocal-matches-polynomial-cofixed-space}.
\end{align}
\end{corollary}
\begin{proof}
Equation \eqref{reciprocal-matches-quotient-cofixed-space} is 
immediate from the discussion surrounding Proposition~\ref{fixed-cofixed-duality}.  
Then \eqref{reciprocal-matches-quotient-cofixed-space}
implies \eqref{reciprocal-matches-polynomial-cofixed-space},
since the isomorphism \eqref{S-and-Q-isomorphic-in-low-degrees}
shows $(S_{P_\alpha})_d \cong (Q_{P_\alpha})_d$ for $0 \leq d \leq q^m-1$.
Lastly \eqref{reciprocal-matches-polynomial-cofixed-space}
implies \eqref{limit-of-reciprocal-matches-polynomial-cofixed-space}.
\end{proof}

\begin{corollary}
\label{first-conj-implies-second-conj-corollary}
Parabolic Conjecture~\ref{mod-powers-conjecture} implies
Parabolic Conjecture~\ref{cofixed-conjecture}.
\end{corollary}
\begin{proof}
Assuming Parabolic Conjecture~\ref{mod-powers-conjecture}, Equation
\eqref{limit-of-reciprocal-matches-polynomial-cofixed-space} implies
$$
\Hilb(S_{P_\alpha},t) =\lim_{m \rightarrow \infty} t^{d_0} \Hilb(Q^{P_\alpha},t^{-1}) 
=\lim_{m \rightarrow \infty} t^{d_0} C_{\alpha,m}(t^{-1}).
$$
Hence Parabolic Conjecture~\ref{cofixed-conjecture} follows once one checks the following
assertion:
\begin{equation}
\label{reciprocal-congruence}
t^{d_0} C_{\alpha,m}(t^{-1}) \equiv
  \sum_{\substack{\beta:
\beta \leq \alpha}}
  t^{\sum_{i=1}^\ell \alpha_i(q^{B_i}-1)}
 \prod_{i=1}^\ell
 \prod_{j=0}^{\beta_i-1}
 \frac{1}
 {1 - t^{q^{B_i}-q^{B_{i-1}+j}}}  \quad \bmod{(t^{q^m})}.
\end{equation}
To prove \eqref{reciprocal-congruence}, one first uses
the definition \eqref{parabolic-Catalan-with-m-definition} of $C_{\alpha,m}(t)$ to do a straightforward calculation showing
\begin{equation}
\label{reciprocal-of-C-explicitly}
t^{d_0} C_{\alpha,m}(t^{-1})=
 \sum_{\substack{\beta:
\beta \leq \alpha\\|\beta|\leq m}}
  t^{\sum_{i=1}^\ell \alpha_i (q^{B_i}-1) }
 \frac{ \prod_{j=0}^{|\beta|-1} (1-t^{q^m-q^j}) }
 { \prod_{i=1}^\ell \prod_{j=0}^{\beta_i-1} (1-t^{q^{B_i}-q^{B_{i-1}+j}}) }.
\end{equation}
Since $q^{B_i}-1 \geq q^{B_i - 1}$, one has
\[
\sum_{i=1}^\ell \alpha_i (q^{B_i}-1)
\geq \sum_{i=1}^\ell \alpha_i q^{B_i - 1}
\geq \alpha_\ell q^{B_\ell - 1}
\geq q^{|\beta| - 1}.
\]
This implies that for each $j=0,1,\ldots,|\beta|-1$ one has 
$(q^m-q^j)+\sum_{i=1}^\ell \alpha_i (q^{B_i}-1) \geq q^m$.
Therefore the right side in \eqref{reciprocal-of-C-explicitly}
is equivalent mod $(t^{q^m})$
to the right side in \eqref{reciprocal-congruence}.
\end{proof}

\begin{corollary}
\label{n=2-equivalence-corollary}
In the bivariate case $n=2$, 
Parabolic Conjectures~\ref{mod-powers-conjecture} and \ref{cofixed-conjecture}
are equivalent.
\end{corollary}
\begin{proof}
Corollary~\ref{first-conj-implies-second-conj-corollary} showed
that Parabolic Conjecture~\ref{mod-powers-conjecture} implies
Parabolic Conjecture~\ref{cofixed-conjecture} 
for any $n$.  The reverse implication when $n=2$ arises when
two coefficient comparisons valid for general $n$ ``meet in the middle", as we now explain.  Again, in this proof, all symbols ``$\equiv$'' mean
congruence mod $(t^{q^m})$.  On one hand, one has
$$
\Hilb(Q^{P_\alpha},t) 
\equiv \Hilb(S^{P_\alpha},t) 
= \prod_{i=1}^\ell \prod_{j=0}^{\alpha_i-1}
\frac{1}{1 - t^{q^{A_i}-q^{A_{i-1}+j}}}
\equiv C_{\alpha,m}(t)
$$
where the left congruence is 
\eqref{invariants-agree-in-low-degree},
the middle equality is 
\eqref{parabolic-invariants-hilbert-series},
and the right congruence is 
Proposition~\ref{equality-up-to-q^m-prop}.  Therefore
$\Hilb(Q^{P_\alpha},t)$ and $C_{\alpha,m}(t)$ have the same coefficients on $1,t,t^2,\ldots,t^{q^m-1}$.
On the other hand, one has 
$$
t^{d_0} \Hilb(Q^{P_\alpha},t^{-1}) 
\equiv \Hilb(S_{P_\alpha},t) 
= \sum_{\substack{\beta:
\beta \leq \alpha}}
  t^{\sum_{i=1}^\ell \alpha_i(q^{B_i}-1)}
 \prod_{i=1}^\ell
 \prod_{j=0}^{\beta_i-1}
 \frac{1}
 {1 - t^{q^{B_i}-q^j}} 
\equiv t^{d_0} C_{\alpha,m}(t^{-1})
$$
where the left congruence is \eqref{reciprocal-matches-polynomial-cofixed-space},
the middle equality is Parabolic Conjecture~\ref{cofixed-conjecture}, 
and the right congruence is 
Corollary~\ref{reciprocal-congruence}. 
Therefore $\Hilb(Q^{P_\alpha},t)$ and $C_{\alpha,m}(t)$ also have 
the same coefficients on 
$t^{d_0},t^{d_0-1},\ldots,t^{d_0-(q^m-1)}$.  
Since
$d_0=n(q^m-1)$, when $n=2$, this means that $\Hilb(Q^{P_\alpha},t), C_{\alpha,m}(t)$ agree on {\it all} coefficients.
\end{proof}

\begin{example}
We illustrate some of the assertions of Corollary~\ref{complementary-degree-corollary}
for $n=m=2$ and $q=3$, where
$$
S=\FF_3[x,y], \qquad 
Q=S/(x^9,y^9), \qquad
d_0=2(3^2-1)=16.
$$
Our results in  Appendix~\ref{n=2-section} 
below show that Conjectures~\ref{mod-powers-conjecture} 
and \ref{cofixed-conjecture} hold for $n=2$.  Therefore for 
the group $G=GL_2(\FF_3) \,\, (=P_{(2)})$, one can compute that
$$
\begin{array}{lll}
\Hilb(S^G,t)=\dfrac{1}{(1-t^6)(1-t^8)} &= 1+ t^6 + t^8 + O(t^9) &\\

\Hilb(Q^G,t)=C_{2,2}(t) &= 1 + t^6 + t^8 + t^{10} + t^{12} + t^{16}
                       & \equiv \Hilb(S^G,t) \bmod{t^9}, 
\end{array}
$$
and similarly
$$
\begin{array}{lll}
\Hilb(S_G,t)=1 + \dfrac{t^{4}}{1-t^2} + \dfrac{t^{16}}{(1-t^6)(1-t^8)} &= 1+ t^4+ t^6 + t^8 + O(t^9) & \\
t^{16}\Hilb(Q^G,t^{-1})&= 1 + t^4 + t^6 + t^8 + t^{10} + t^{16}
                       &\equiv \Hilb(S_G,t) \bmod{t^9}.
\end{array}
$$
Note that $\Hilb(Q^G,t)$ is {\it not} a reciprocal polynomial in $t$, that is, its coefficient sequence is not symmetric.  
In particular, although the ring $Q$ is Gorenstein, its $G$-fixed subalgebra $Q^G$ is {\it not}.
\end{example}

\section{The case where $m$ is at most $1$}
\label{small-m-section}

When $m=0$, Parabolic Conjecture~\ref{mod-powers-conjecture}
says little: $Q=S/\mm^{[q^0]}=S/\mm=\FF_q$ has no variables, 
so $Q^{P_\alpha}=Q=\FF_q$ and $\Hilb(Q^{P_\alpha},t)=1$.  Meanwhile,
$C_{\alpha,0}(t)=1$ since \eqref{parabolic-Catalan-with-m-definition}
has only the $\beta=0$ summand.  

 The $m=1$ case is less trivial.

\begin{proposition}
\label{small-m-proposition}
Parabolic Conjecture~\ref{mod-powers-conjecture} holds for $m=1$.
\end{proposition}
\begin{proof}
Given the composition $\alpha=(\alpha_1,\ldots,\alpha_\ell)$ of $n$,
the only weak compositions $\beta$ with $0 \leq \beta \leq \alpha$ and
$|\beta| \leq m=1$ are $\beta=0$ and $\beta=e_k=(0,\ldots,0,1,0,\ldots,0)$
for $k=1,2,\ldots,\ell$.  One therefore finds that
$$
C_{\alpha,1}(t)
=t^{e(1,\alpha, 0)} \qbin{1}{0,1}{q,t} +
\sum_{k=1}^\ell t^{e(1,\alpha, e_k)} \qbin{1}{e_k,0}{q,t} 
= t^{n(q-1)} + \sum_{k=1}^\ell t^{A_{k-1} (q-1)} 
= \sum_{k=0}^\ell t^{A_k(q-1)},
$$
recalling that $A_\ell=n$ and the convention that $A_0=0$.
Thus to show $C_{\alpha,1}(t)=\Hilb( Q^{P_{\alpha}}, t)$,
it will suffice to show that $Q^{P_{\alpha}}$
has $\FF_q$-basis given by the images of the monomials
\begin{equation}
\label{square-free-parabolic-invariant-monomials}
\{ (x_1 x_2 \cdots x_{A_k})^{q-1} \}_{k=0,1,\ldots,\ell}.
\end{equation}
To argue this, consider any polynomial
$$
f(\xx)=\sum_{\substack{a=(a_1,\ldots,a_n) 
               \\ a_i \in \{0,1,\ldots,q-1\}}} c_a \xx^a
$$
representing an element of the quotient 
$Q=S/\mm^{[q]}$.
One has that $f(\xx)$ is invariant 
under the {\it diagonal matrices} $T$ inside $P_\alpha$
if and only if each entry $a_i$ is either $0$ or $q-1$, that is, if $f(\xx)$ has
the form
\begin{equation}
\label{square-free-monomial-expression}
f(\xx) = \sum_{A \subset \{1,2,\ldots,n\}} c_A \,\, \xx_A^{q-1} 
\end{equation}
where $\xx_A:=\prod_{j \in A} x_j$, so that $\xx_A^{q-1} = \prod_{j \in A} x_j^{q-1}$.

We claim that such an $f$ is 
furthermore invariant under the {\it Borel subgroup} $B$
of upper triangular matrices if and only if 
each monomial $\xx_A^{q-1}$ in the support of $f$
has $A$ forming an initial segment $A=\{1,2,\ldots,k\}$ for some $k$.
To see this claim, note that $B$ is generated by $T$ together with 
$\{ u_{ij}:  1 \leq i < j \leq n\}$
where $u_{ij}$  sends $x_j \mapsto x_j + x_i$ and fixes 
all other variables $x_\ell$ with $\ell \neq j$.  Working mod $\mm^{[q]}$ one checks that
$$
u_{i,j}( \xx_A^{q-1} ) =
\begin{cases}
\xx_A^{q-1} & \text{ if }\{i,j\} \cap A \neq \{j\},\\
\xx_A^{q-1} + \xx_{A \setminus \{j\} \cup \{i\}}^{q-1} 
  & \text{ if }\{i,j\} \cap A = \{j\}.
\end{cases}
$$
From this it is easily seen that each monomial $(x_1 x_2 \cdots x_k)^{q-1}$ has $B$-invariant
image in $Q$.  On the other hand,  if $f(\xx)$ as 
in \eqref{square-free-monomial-expression} has $c_A \neq 0$ for
some $A$ which is not an initial segment,
then there exists $1 \leq i < j \leq n$ for which
$\{i,j\} \cap A = \{j\}$, and one finds 
that $u_{i,j}(f) \neq f$, since
$u_{i,j}(f)-f$ has coefficient $c_A$ on 
$\xx_{A \setminus \{j\} \cup \{i\}}^{q-1}$.

Lastly, an element of this more specific form
$f(\xx) = \sum_{k=0}^n c_k (x_1 x_2 \cdots x_k)^{q-1}$ 
will furthermore be invariant under the
subgroup $\Symm_{\alpha_1} \times \cdots \times \Symm_{\alpha_\ell}$
of block permutation matrices inside $P_\alpha$ if and only if
it is supported on the monomials in
\eqref{square-free-parabolic-invariant-monomials}.
Since $P$ is generated by the Borel subgroup $B$ together with
this subgroup $\Symm_{\alpha_1} \times \cdots \times \Symm_{\alpha_\ell}$,
the monomials in \eqref{square-free-parabolic-invariant-monomials}
give an $\FF_q$-basis for $Q^{P_\alpha}$.
\end{proof}

\section{The cofixed quotient $S_G$ as an $S^G$-module}
\label{module-over-invariants-section}

Note that Parabolic Conjecture~\ref{cofixed-conjecture} has the following two consequences for 
the rational function 
$\frac{\Hilb(S_{P_\alpha},t)}{\Hilb(S^{P_\alpha},t)}$:
\begin{eqnarray}
\label{denominator-bound}
&\displaystyle \frac{\Hilb(S_{P_\alpha},t)}{\Hilb(S^{P_\alpha},t)} & \text{ lies in } \ZZ[t], \text{ and }\\
\label{rank-one-limit-consequence}
\lim_{ t \rightarrow 1 } & \displaystyle \frac{\Hilb(S_{P_\alpha},t)}{\Hilb(S^{P_\alpha},t)}& = 1.
\end{eqnarray}
The goal of the subsections below is to explain why 
\eqref{denominator-bound}, \eqref{rank-one-limit-consequence} do indeed hold,
essentially due to three facts:
\begin{itemize}
\item[(1)]
the $P_\alpha$-cofixed quotient $S_{P_\alpha}$ is a finitely generated module over
the $P_\alpha$-invariant ring $S^{P_\alpha}$;
\item[(2)]
while $S_{P_\alpha}$ is {\it not} in general a free $S^{P_\alpha}$-module, 
it does always have $S^{P_\alpha}$-rank one; and
\item[(3)]
the $P_\alpha$-invariant ring $S^{P_\alpha}$ is polynomial, as shown in
\cite{Hewett,Mui}.
\end{itemize}

\subsection{The cofixed spaces as a module over fixed subalgebra}

Facts (1), (2) above hold generally for
finite group actions, and are analogous to 
well-known facts about invariant rings.
As we have not found them
in the literature, we discuss them here.

\begin{proposition}
\label{proposition: simple basis}
Fix a field $k$, a $k$-algebra $R$, an $R$-module $M$,
and let $G$ be any subgroup of $\Aut_R(M)$,
the $R$-module automorphisms of $M$.
Then one has that
\begin{enumerate}
\item[(i)]
the $k$-linear span $N$ of all elements 
$\{g(m)-m\}_{g \in G, m \in M}$ is an $R$-submodule of $M$, and hence
\item[(ii)]
the cofixed space $M_G:=M/N$ is a quotient $R$-module of $M$.
\end{enumerate}
Furthermore, if $\{ m_i\}_{i \in I}$ generate $M$ as an $R$-module, and
if $\{g_j\}_{j \in J}$ generate $G$ as a group, then
\begin{enumerate}
\item[(iii)]
the images
$\{\overline{m}_i\}_{i \in I}$ generate $M_G$ as an $R$-module, and
\item[(iv)]
the elements 
$\{g_j^{\pm 1}(m_i)-m_i\}_{i \in I,  j \in J}$ 
generate $N$ as an $R$-module.
\end{enumerate}
\end{proposition}
\begin{proof}
All assertions are completely straightforward, except possibly for
(iv), which relies on this calculation:
$$
g_1 g_2(m) - m = g_1 g_2(m) -g_2(m) + g_2(m) - m
$$
and the hypotheses let one express $g_2(m) = \sum_{i \in I} r_i m_i$ for some $r_i$ in $R$, so that one can rewrite this as
\[
g_1 g_2(m) - m = \sum_{i \in I} r_i (g_1 (m_i) - m_i) + (g_2(m) - m). 
\qedhere
\]
\end{proof}

\begin{corollary}
Let $S$ be a finitely generated 
$k$-algebra and $G$ a finite subgroup of $k$-algebra
automorphisms of $S$, e.g., $S=k[x_1,\ldots,x_n]$ and $G$ a finite subgroup
of $GL_n(k)$ acting by linear substitutions.

Then the $G$-cofixed space $S_G$ is a finitely generated module 
over the $G$-fixed subalgebra $S^G$.
\end{corollary}
\begin{proof}
Via Proposition~\ref{proposition: simple basis}(ii,iii),
it suffices to show that $S$ is a finitely generated $S^G$-module.
This is well-known argument via \cite[Cor. 5.2]{AtiyahMacdonald};
see \cite[Thm. 1.3.1]{Benson}, \cite[Thm. 2.3.1]{Smith}. One has
that $S$ is integral over $S^G$, as any $x$ in $S$ satisfies the
monic polynomial $\prod_{g \in G}(t-g(x))$ in $S^G[t]$, and
$S$ is finitely generated as an algebra over $S^G$
because it is finitely generated as a $k$-algebra.
\end{proof}

\begin{example}
In the case of $M=S=\FF_q[x_1,\ldots,x_n]$ and
$G=GL_n(\FF_q)$, one has that $S$ is even
a {\it free} $S^G$-module of rank $|G|$ with
an explicit $S^G$-basis of monomials 
$
\{ x^\alpha \}_{0 \leq \alpha_i \leq q^n-q^{i-1}-1}
$
provided by Steinberg \cite{Steinberg} in his proof of Dickson's Theorem.
Consequently, $S_G$ is generated by the images of these monomials,
and Proposition~\ref{proposition: simple basis}(iv)
leads to an explicit finite presentation of $S_G$ as a 
quotient of the free $S^G$-module $S$, useful for computations.
\end{example}

\begin{corollary}
When a finite subgroup $G$ of $GL_n(k)$ acting by
linear substitutions on $S=k[x_1,\ldots,x_n]$ has
$G$-fixed subalgebra $S^G$ which is again a polynomial algebra,
then $\Hilb(S_G,t)/\Hilb(S^G,t)$ lies in
$\ZZ[t]$.
\end{corollary}
\begin{proof}
When $S^G$ is polynomial,
the {\it Hilbert syzygy theorem} (see e.g.\ \cite[\S 2.1]{Benson}, 
\cite[\S 6.3]{Smith}) implies that $S_G$ will have a finite
$S^G$-free resolution 
$
0 \rightarrow F_n \rightarrow \cdots \rightarrow F_1 \rightarrow F_0 \rightarrow S_G \rightarrow 0
$
where $F_i = \bigoplus_{j \geq 0} S^G(-j)^{\beta_{i,j}}$ for some nonnegative integers $\beta_{i,j}$.
Here $R(-j)$ denotes a copy of the graded ring $R$, regarded as a module over itself, but
with grading shift so that the unit $1$ is in degree $j$, so that 
$
\Hilb(F_i,t)=\Hilb(S^G,t) \cdot \sum_{j \geq 0} \beta_{i,j} t^j.
$
Considering Euler characteristics in each homogeneous component of the 
resolution gives
$$
\Hilb(S^G,t) \sum_{i ,j \geq 0} (-1)^i \beta_{i,j} t^j 
= \Hilb(S_G,t)
$$  
so that $\Hilb(S_G,t)/\Hilb(S^G,t) = \sum_{i ,j \geq 0} (-1)^i \beta_{i,j} t^j$ 
lies in $\ZZ[t]$.
\end{proof}

\subsection{The cofixed space is a rank one module}

We next explain, via consideration of the rank of $S_G$ as an $S^G$-module,
why one should expect \eqref{rank-one-limit-consequence} to hold.

\begin{definition}
Recall \cite[\S 12.1]{DummitFoote}
for a finitely generated $M$ over an integral domain $R$, 
that $\rank_R(M)$ is the maximum size of an $R$-linearly independent
subset of $M$.
\end{definition}

Alternatively, $\rank_R(M)$ is the largest
integer $r$ such that $M$ contains a free $R$-submodule $R^r$, 
and in this situation, the quotient
$M/R^r$ will be all {\it $R$-torsion}, that is, for every $x$ in $M/R^r$
there exists some $a \neq 0$ in $R$ with $ax=0$.
One can equivalently define this using
the {\it field of fractions} $K := \Frac(R)$ via
\begin{equation}
\label{rank-as-frac-dimension}
\rank_R(M):=\dim_{K}\left(  K \otimes_R M \right).
\end{equation}
Indeed, clearing denominators shows that a subset $\{m^{(i)}\} \subset M$ is $R$-linearly independent 
if and only if $\{1 \otimes m^{(i)}\} \subset  K \otimes_R M$
is $K$-linearly independent.
 
In the graded setting, one has the following well-known characterization of rank via Hilbert series.
\begin{proposition}
\label{graded-module-rank-as-a-limit}
For $R$ an integral domain which is also a 
finitely generated graded $k$-algebra,
and $M$ a finitely generated graded $R$-module, 
the rational functions $\Hilb(R,t)$ and $\Hilb(M,t)$ satisfy
\[
\rank_R(M) = \lim_{t \rightarrow 1} \frac{\Hilb(M,t)}{\Hilb(R,t)}.
\]
\end{proposition}
\begin{proof}
Letting $r:=\rank_R(M)$, we claim that 
one can choose an $R$-linearly independent subset of size $r$
in $M$ consisting of {\it homogeneous} elements as follows.
Given {\it any} $R$-linearly independent subset $\{m^{(i)}\}_{i=1,2\ldots,r}$,
decompose them into their homogeneous components $m^{(i)}=\sum_j m^{(i)}_j$.
Then the set of all such components  $\{m^{(i)}_j\}$ spans an 
$R$-submodule of $M$ containing the $R$-submodule spanned by
$\{m^{(i)}\}_{i=1,2,\ldots,r}$.  Thus the set of all such components must
contain an $R$-linearly independent subset of size $r$.

Now consider the free $R$-submodule $R^r:=\oplus_{i=1}^r Rm_i$
spanned by a homogeneous $R$-linearly independent subset $\{m_i\}_{i=1,2,\ldots,r}$, 
so that the quotient $M/R^r$ will be all $R$-torsion.  Then
$$
\lim_{t \rightarrow 1} \frac{\Hilb(M,t)}{\Hilb(R,t)}
 = \lim_{t \rightarrow 1}  \frac{\Hilb(R^r,t)}{\Hilb(R,t)} 
  +  \lim_{t \rightarrow 1}\frac{\Hilb(M/R^r,t)}{\Hilb(R,t)}.
$$
Since $\Hilb(R^r,t) / \Hilb(R,t) = \sum_{i=1}^r t^{\deg(m_i)}$,
the first limit on the right is $r$.
One argues that the second limit on the right vanishes as follows.
Assume $R$ has Krull dimension $d$, that is, $\Hilb(R,t)$ 
has a pole of order $d$ at $t=1$.  Thus one must show that
$\Hilb(M/R^r,t)$ has its pole of order at most $d-1$.
To this end, choose homogeneous 
generators $y_1,\ldots,y_N$ for the $R$-torsion module $M/R^r$,
say with $\theta_i y_i=0$ for nonzero homogeneous $\theta_i$ 
in $R$.  Then one has a graded $R$-module surjection
$
\bigoplus_{i=1}^N R/(\theta_i) (-\deg(y_i))
\twoheadrightarrow M/R^r
$
sending the basis element of $R/(\theta_i)$ to $y_i$.
This gives a coefficientwise inequality 
\begin{equation}
\label{coefficientwise-inequality}
\Hilb(M/R^r,t) \leq \sum_{i=1}^N t^{\deg(y_i)} \Hilb(R/(\theta_i),t)
 = \sum_{i=1}^N (1-t^{\deg(\theta_i)})  t^{\deg(y_i)} \Hilb(R,t)
\end{equation}
between power series with nonnegative coefficients, which are
also rational functions having poles confined to the unit circle. 
As each summand on the right of \eqref{coefficientwise-inequality}
has a pole of order at most $d-1$ at $t=1$,
the same holds for $\Hilb(M/R^r,t)$.
\end{proof}

For a subgroup $G$ of ring automorphisms of the
domain $S$, denote by $K:=\Frac(S)^G$ the $G$-invariant subfield of $L:=\Frac(S)$.
When $G$ is finite, an easy argument  \cite[Prop. 1.1.1]{Benson}, 
\cite[Prop. 1.2.4]{Smith} shows that 
$$
K:=\Frac(S^G) = \Frac(S)^G \left( =L^G \right),
$$
giving this commuting diagram of inclusions
\begin{equation}
\label{invariant-ring-field-inclusions}
\xymatrix{
S \ar@{^{(}->}[r] &L \\
S^G \ar@{^{(}->}[r] \ar@{^{(}->}[u]&K\ar@{^{(}->}[u]
}
\end{equation}
Consequently, Proposition~\ref{graded-module-rank-as-a-limit}
together with the next result immediately imply 
\eqref{rank-one-limit-consequence}.

\begin{proposition}
\label{rank-one-prop}
A finite group $G$ of automorphisms of an integral domain $S$
has $\rank_{S^G} S_G=1$.
\end{proposition}
\begin{proof}
Using \eqref{rank-as-frac-dimension} to characterize rank,
it suffices to show this chain of three $K$-vector space isomorphisms:
\begin{equation}
\label{three-step-isomorphism}
K \otimes_{S^G} S_G 
\quad \cong \quad L_G  
\quad \cong \quad (KG)_G 
\quad \cong \quad K.
\end{equation}
For the first step in \eqref{three-step-isomorphism}, start with the short exact sequence that defines $S_G$
$$
0 \longrightarrow
\sum_{\substack{g \in G\\s \in S}} S^G (g(s)-s)  \longrightarrow
S \longrightarrow
S_G \longrightarrow
0
$$
and apply the exact localization functor $K \otimes_{S^G}(-)$ to give the short exact sequence
\begin{equation}
\label{localized-exact-sequence}
0 \longrightarrow
\sum_{\substack{g \in G\\ s \in S}} K \otimes_{S^G}  S^G(g(s)-s)  \longrightarrow
K \otimes_{S^G} S  \longrightarrow
K \otimes_{S^G} S_G \longrightarrow
0
\end{equation}
Using the $K$-vector space isomorphism 
$
K \otimes_{S^G} S \cong  L
$
induced by 
$f \otimes s \longmapsto fs$, the sequence \eqref{localized-exact-sequence} becomes
$$
0 \longrightarrow
\sum_{\substack{g \in G\\ f \in L}} K (g(f)-f)  \longrightarrow
L  \longrightarrow
K \otimes_{S^G} S_G \longrightarrow
0
$$
which shows that $K \otimes_{S^G} S_G \cong L_G$, completing the first step.

The second step in \eqref{three-step-isomorphism} comes from considering the Galois extension 
$
K=L^G \hookrightarrow L
$
having Galois group $G$, that appears as the right vertical map in \eqref{invariant-ring-field-inclusions}.  
The Normal Basis Theorem of Galois Theory \cite[Theorem 13.1]{Lang} asserts that,
not only is $L \cong K^{|G|}$ as a $K$-vector space, 
but $L$ is even isomorphic to the {\it left-regular representation} $K G$ 
as $K G$-module.  Hence $L_G \cong (K G)_G$, completing the second step.

The third step in \eqref{three-step-isomorphism} comes from the short
exact sequence of $KG$-modules
\begin{equation}
\label{augmentation-exact-sequence}
0 \longrightarrow I_G  \longrightarrow
KG  \overset{\epsilon}{\longrightarrow}
K \longrightarrow 0.
\end{equation}
Here $G$ acts trivially on $K$, while the {\it augmentation ideal} $I_G$
is the kernel of the {\it augmentation map} $\epsilon$ sending each $K$-basis element $g$ of $KG$ to $1$ in $K$.  Since $I_G$ is $K$-spanned by $g - h$ for $g,h$ in $G$,
the sequence \eqref{augmentation-exact-sequence} shows that $(K G)_G \cong K$, completing the third step.
\end{proof}

This immediately implies the following corollary, explaining \eqref{rank-one-limit-consequence}.

\begin{corollary}
When a finite subgroup $G$ of $GL_n(k)$ acting by
linear substitutions on $S=k[x_1,\ldots,x_n]$ has
$G$-fixed subalgebra $S^G$ which is again a polynomial algebra,
then 
$$
\lim_{t \rightarrow 1} \frac{\Hilb(S_G,t)}{\Hilb(S^G,t)}=\rank_{S^G} S_G =1.
$$
\end{corollary}

\subsection{On the module structure 
of the $G$-cofixed space} 

For the full general linear group $G=GL_n(\FF_q)$,
the structure of $S_G$ as an $S^G$-module exhibited for $n=1$ 
in Example~\ref{small-n-example}
and for $n=2$ in Theorem~\ref{bivariate-Stanley-decomposition} 
suggests a general question.

Recall that $S^G=\FF_q[D_{n,0},D_{n,1},\ldots,D_{n,n-1}]$
where the Dickson polynomials $D_{n,i}$ were defined in the introduction.  
Consider subalgebras of $S^G$ defined for $i=1,2,\ldots,n$ by
\[
\FF_q[Z_i]:=\FF_q[D_{n,n-i},D_{n,n-i+1},\ldots,D_{n,n-2},D_{n,n-1}].
\]

\begin{question}
\label{Stanley-decomposition-conjecture}
Does there exist a subset $M$ of homogeneous
elements minimally generating $S_G$ as an $S^G$-module, with
a decomposition $M =\bigsqcup_{i=1}^n M_i$ having the following 
properties?
\begin{itemize}
\item
The $\FF_q[Z_i]$-submodule generated by $M_i$ within  $S_G$
is $\FF_q[Z_i]$-free.
\item
The Dickson polynomials $D_{n,0},D_{n,1},\ldots,D_{n,n-i-1}$ not
in $\FF_q[Z_i]$ all annihilate every element of $M_i$.
\item 
The last set $M_n$ is a singleton, whose unique  
element has degree $(n-1)(q^n-1)$.
\end{itemize}
\end{question}

\begin{example}
In the case $n=1$, Example~\ref{small-n-example} shows that
$S_G$ is a free $S^G$-module of rank $1$ with basis element given
by the image of $1$.  This answers 
Question~\ref{Stanley-decomposition-conjecture} affirmatively by 
setting $M=M_1:=\{1\}$.
\end{example}

\begin{example}
\label{Stanley-decomposition-example-for-n=2}
In the case $n=2$, Theorem~\ref{bivariate-Stanley-decomposition} below
will answer Question~\ref{Stanley-decomposition-conjecture} 
affirmatively by setting $M=M_1 \sqcup M_2$, where 
$
M_1:=\{1,XY,X^2Y,\ldots,X^{q-2}Y\}
$
and
$
M_2:=\{X^q Y\},
$
with $X:=x^{q-1}, Y:=y^{q-1}$.
\end{example}

Before discussing the case $n=3$ in further detail, we mention 
a general recurrence for the power series
$$
f_n(t):=\sum_{k = 0}^n t^{n(q^k - 1)}\prod_{i=0}^{k-1} \frac{1}{1 - t^{q^k - q^i}}
$$
that was conjectured to equal $\Hilb(S_G,t)$ in 
Conjecture~\ref{cofixed-conjecture}. An easy calculation shows that
\begin{equation}
\label{Dennis's-recurrence}
f_n(t)=\left( f_{n-1}(t)-t^{(n-1)(q-1)} f_{n-1}(t^q) \right) +  
  \dfrac{t^{(n-1)(q^n - 1)}}{\prod_{i=0}^{n-1}( 1 - t^{q^n - q^i} )}.
\end{equation}

\noindent
An affirmative answer to Question~\ref{Stanley-decomposition-conjecture} would 
interpret the two summands on the right in \eqref{Dennis's-recurrence} as follows:
\begin{itemize}
\item 
the last summand on the right 
in \eqref{Dennis's-recurrence} 
would be the Hilbert series for the $S^G$-submodule of $S_G$ 
generated by the singleton set $M_n$, while
\item
the difference 
$f_{n-1}(t)-t^{(n-1)(q-1)} f_{n-1}(t^q)$ would be the Hilbert series 
for the $S^G$-submodule generated by $M_1 \sqcup \cdots \sqcup M_{n-1}$,
or alternatively, the kernel of multiplication by $D_{n,0}$ on
$S_G$.
\end{itemize}
Somewhat suggestively, it can be shown directly that the difference
$f_{n-1}(t)-t^{(n-1)(q-1)} f_{n-1}(t^q)$ has nonnegative coefficients as
a power series in $t$;  we omit the details of this proof.

\begin{example}
\label{Stanley-decomposition-example-for-n=3}
In the case $n=3$, the recurrence \eqref{Dennis's-recurrence} 
suggests a more precise version of 
Question~\ref{Stanley-decomposition-conjecture}
that agrees with computer experiments. 
Example~\ref{Stanley-decomposition-example-for-n=2}
shows that
$$
f_2(t)=\frac{1+t^{2(q-1)}[q-2]_{t^{q-1}}}{1-t^{q^2-q}}
           +  \frac{t^{q^2-1}}{(1-t^{q^2-q})(1-t^{q^2-1})} 
=:m_{2,1}(t)  +  m_{2,2}(t),
$$
using the notation $[n]_t:=1+t+\cdots+t^{n-1}$.
Then the recurrence \eqref{Dennis's-recurrence} applied with $n=3$ gives
\begin{equation}
\label{n=3-calculation-with-recurrence}
\begin{array}{rcccccc}
f_3(t)
&=&\left( m_{2,1}(t)-t^{2(q-1)} m_{2,1}(t^q) \right)
   &+&\left( m_{2,2}(t)-t^{2(q-1)} m_{2,2}(t^q) \right)
   &+& \displaystyle \frac{t^{2(q^3-1)}}{(1-t^{q^3-q^2})(1-t^{q^3-q})(1-t^{q^3-1})} \\
&=& \displaystyle \frac{A_1(t)}{1-t^{q^3-q^2}} 
    &+&  \displaystyle \frac{A_2(t)}{(1-t^{q^3-q^2})(1-t^{q^3-q})}
    &+&  \displaystyle \frac{A_3(t)}{(1-t^{q^3-q^2})(1-t^{q^3-q})(1-t^{q^3-1})}
\end{array}
\end{equation}
where one can compute these numerators explicitly:
\begin{equation}
\label{n=3-numerator-expressions}
\begin{aligned}
A_1(t) &= [q]_{t^{q^2-q}} + t^{2(q-1)}  ([q-2]_{t^{q-1}} (1+t^{q^2-q}) - 1 ) +
             t^{(2q+3)(q-1)} [q-2]_{t^{q^2-q}}  [q-3]_{t^{q-1}}, 
 \\
A_2(t) &= t^{q^2-1} [q]_{t^{q^2-1}} [q]_{t^{q^2-q}} -  t^{(q-1)(q^2+q+2)}, \\
A_3(t) &=t^{2(q^3-1)}.
\end{aligned}
\end{equation}
Note that $A_1(t),A_2(t)$ are polynomials in $t$ 
with nonnegative coefficients\footnote{Nonnegativity of $A_1(t), A_2(t)$ 
is manifest from
\eqref{n=3-numerator-expressions} for $q \geq 3$;  for $q=2$ it also holds,
although it is less clear from \eqref{n=3-numerator-expressions}.}.
The following conjecture has been checked by D. Stamate for $n=3$ and
$q=2,3,4,5$ using {\tt Singular}.

\begin{conjecture}
Question~\ref{Stanley-decomposition-conjecture} for $n=3$ has
affirmative answer with
$
\sum_{f \in M_i} t^{\deg(f)} = A_i(t) 
$
as in \eqref{n=3-numerator-expressions}.
\end{conjecture}

\end{example}

We close this section with some remarks on 
Question~\ref{Stanley-decomposition-conjecture}, some providing evidence
on the affirmative side, and some on the negative side.

\begin{remark}
If the answer to Question~\ref{Stanley-decomposition-conjecture} is affirmative,
then it would imply that the Dickson polynomial of lowest degree 
$
D_{n,n-1}
$
acts on the $S^G$-module $S_G$ as a nonzero divisor.
One can check that this property does indeed hold for $D_{n,n-1}$
via the same argument of Karagueuzian and Symonds 
\cite[Lemma 2.5]{KaragueuzianSymonds1} used
in Proposition~\ref{Karagueuzian-Symonds-prop} below.  The key fact is
that Dickson's expression for $D_{n,n-1}$ as a quotient of determinants 
shows it to be a homogeneous polynomial in $x_1,\ldots,x_n$ of
degree $q^n-q^{n-1}$ with $x_n^{q^n-q^{n-1}}$ as its leading monomial
in $x_n$.
\end{remark}

\begin{remark}
Recurrence~\eqref{Dennis's-recurrence} with $n=4$ gives rise,
via a calculation similar to \eqref{n=3-calculation-with-recurrence},
to an expression
$$
f_4(t)=  \sum_{i=1}^{4} \frac{B_i(t)}{\prod_{j=1}^{i} (1-t^{q^4-q^{4-j}}) }.
$$
However, one finds that for $q=2$, the numerator $B_1(t)$ is equal to $1 - t^3 + t^4$,
which has a negative coefficient.  Analogous calculations for higher
values of $n$ and small values of $q$ yield similar negative coefficients in 
the other numerator terms.
\end{remark}

\begin{remark}
One might ask why Question~\ref{Stanley-decomposition-conjecture}
has  been formulated only for $G$, and not
for all parabolic subgroups $P_\alpha$ of $G$.  In fact, 
Theorem~\ref{Stanley-decomposition-for-B-with-n=2} below does prove
such a result for $n=2$, when there is only one proper
parabolic subgroup, the Borel subgroup $B=P_{(1,1)}$ inside
$G=GL_2(\FF_q)$.

However, computer calculations in {\tt Sage} suggest that a
naive formulation of such a question has a negative answer in general.
Specifically, for $n=3$ and $q=4$
with $B=P_{(1,1,1)}$ inside $G=GL_3(\FF_4)$, one
encounters the following difficulty.
One wants a minimal generating set $M$ for $S_B$ as an $S^B$-module
of a particular form.
Note that here 
$S^B=\FF_4[f_3,f_{12},f_{48}]$, 
where
$
f_3:=x^3
$,
$
f_{12}:=\prod_{c \in \FF_4} (y+cx)^3
= y^{12}+ x^3 y^9 + x^6 y^6 + x^9 y^3,
$
and
$
f_{48}:= D_{3,2}(x,y,z).
$
One can also show,  using the
idea in \cite[\S 2.1]{KaragueuzianSymonds1} and
Proposition~\ref{Karagueuzian-Symonds-prop} below,
that $f_{48}$ acts as a nonzero-divisor on $S_B$. 
Thus one might expect a decomposition of the minimal generators as 
\begin{equation}
\label{chimerical-Stanley-decomposition}
M=M_1 \sqcup M_2 \sqcup M_3 \sqcup M_4
\end{equation}
in which 
\begin{itemize}
\item $\FF_4[f_3, \,\, f_{12}, \,\, f_{48}]$ acts freely on $M_4$,
\item $\FF_4[f_{12}, \,\, f_{48}]$ acts freely on $M_3$, but $f_3$ annihilates it,
\item $\FF_4[f_3, \,\, f_{48}]$ acts freely on $M_2$, but $f_{12}$ annihilates it, and
\item $\FF_4[f_{48}]$ acts freely on $M_1$, but both $f_3, \,\, f_{12}$ annihilate it. 
\end{itemize}
We argue this is impossible as follows.
Let $(S_B)_{\leq d}:=\bigoplus_{i=0}^d (S_B)_d$, and
similarly $(S_B)_{< d}:=\bigoplus_{i=0}^{d-1} (S_B)_d$.  Given a subset $A \subset S_B$, let $S^B A$ denote the $S^B$-submodule of $S_B$ that it generates.  Then
computations show 
$
\Hilb( S^B (S_B)_{\leq 42}, t)
-\Hilb( S^B (S_B)_{<42}, t)
$
agrees up through degree $90$ with 
$$
t^{42} \cdot \Hilb(\, \FF_4[f_3,f_{48}]\, , \,t)  
\,\, +  \,\,
t^{42} \cdot \Hilb(\, \FF_4[f_{12},f_{48}]\, ,\, t).
$$
One can check that this would force that in any decomposition
\eqref{chimerical-Stanley-decomposition}, the sets $M_2, M_3$ must each
contain exactly one element of degree $42$.  But computations show that for 
every element $f$ in $(S_B)_{42}$, the difference
$$
\Hilb\left( S^B \left((S_B)_{<42} \cup \{f\}\right), t \right)
-\Hilb\left( S^B (S_B)_{<42}, t \right)
$$
is neither equal to $t^{42} \Hilb(\FF_4[f_3,f_{48}],t)$, nor to
$t^{42}  \Hilb(\FF_4[f_{12},f_{48}],t).$
Thus there are no suitable choices for these elements of $M_2, M_3$.
%

\end{remark}

\begin{remark}
Question~\ref{Stanley-decomposition-conjecture}
is reminiscent of the {\it Landweber-Stong Conjecture} in 
modular invariant theory, proven
when $q=p$ is prime by Bourguiba and Zarati \cite{BourguibaZarati}:

\begin{conjecture}[Landweber and Stong \cite{LandweberStong}]
For a subgroup $H$ of $GL_n(\FF_q)$ acting on $S=\FF_q[\xx]$,
the depth of the $H$-invariant ring $S^H$ is the maximum $i$ for which
the elements 
$
D_{n,n-i},D_{n,n-i+1},\ldots,D_{n,n-2},D_{n,n-1}
$ 
form a regular sequence on $S^H$.
\end{conjecture}
\end{remark}

\section{Comparing two representations}
\label{CSP-section}

This section reveals the original motivation for our conjectures, analogous to questions on real and complex reflection groups $W$, their {\it parking spaces}, {\it $W$-Catalan numbers}, and {\it Fuss-Catalan} generalizations.  We refer the reader to \cite{ArmstrongRhoadesR, Rhoades} for the full story on this analogy; see also Section~\ref{RCA-question} below.
Roughly speaking, we start by examining two strikingly similar 
$G$-representations, that we will call the 
{\it graded} and {\it ungraded $G$-parking spaces}.  
Parabolic Conjecture~\ref{mod-powers-conjecture} turns out to yield a comparison
of their $P_\alpha$-fixed subspaces.

\subsection{The graded and ungraded $GL_n(\FF_q)$-parking spaces}

\begin{definition} 
For a field $k \supset \FF_q$, 
the {\it graded parking space} for $G=GL_n(\FF_q)$ over $k$ is
$$
Q_k:=k \otimes_{\FF_q} Q 
= k[x_1,\ldots,x_n]/(x_1^{q^m},\ldots,x_n^{q^m}) = S_k /\mm^{[q^m]}
$$
where $S_k:=k[x_1,\ldots,x_n]$ and $\mm:=(x_1,\ldots,x_n)$.
The group $G=GL_n(\FF_q) \subset GL_n(k)$ acts on $S_k$ via linear substitutions, and
also on $Q_k$, just as before.  Thus $Q_k$ is a graded $kG$-module.
\end{definition}

\begin{definition}
For a field $k \supset \FF_q$, the {\it ungraded parking space} 
$$
\ugpark{k} :=\spanof_k \{ e_v: v \in \FF_{q^m}^n \}
$$
for $G$ over $k$ is the $G$-permutation
representation on the points of $\FF_{q^m}^n$ via the embedding
$G=GL_n(\FF_q) \subset GL_n(\FF_{q^m})$,
considered as a $kG$-module.
In other words, the element $g$ in $G=GL_n(\FF_q)$ represented by
a matrix $(g_{ij})$ will send the $k$-basis element $e_v$ 
indexed by $v=(v_1,\ldots,v_n)$ in $\FF_{q^m}^n$ to
$g(e_v) = e_{g(v)}$, where $g(v)_i= \sum_{j=1}^n g_{ij} v_j$.
\end{definition}

\begin{example}
\label{ungraded-parking-space-example}
When $q=3, n=2, m=1$, the ungraded parking space has these nine $k$-basis elements:
$$
\left\{ 
\begin{matrix}
 e_{(-1,+1)}, & e_{(0,+1)}, & e_{(+1,+1)}, \\
 e_{(-1,0)}, & e_{(0,0)}, & e_{(+1,0)}, \\
 e_{(-1,-1)}, & e_{(0,-1)}, & e_{(+1,-1)}
\end{matrix}
\right\}.
$$
For example, $-I_{2 \times 2}$ in $G=GL_2(\FF_3)$ fixes 
$e_{(0,0)}$ and swaps the remaining basis elements as follows:
$$
\begin{array}{rcl}
 e_{(-1,0)} &\leftrightarrow & e_{(+1,0)} \\
 e_{(-1,+1)} &\leftrightarrow & e_{(+1,-1)} \\
 e_{(0,+1)} &\leftrightarrow & e_{(0,-1)} \\
 e_{(+1,+1)} &\leftrightarrow & e_{(-1,-1)}. \\
\end{array}
$$
\end{example}

Note that both $kG$-modules 
$Q_k$ and $\ugpark{k}$ have dimension $(q^m)^n$.
Before investigating their further similarities, we first
note that they are {\it not in general isomorphic} for $n \geq 2$.

\begin{example}
As in Example~\ref{ungraded-parking-space-example}, take $q=3, n=2, m=1$.  One can argue that 
$$
Q_k=k[x_1,x_2]/(x_1^3,x_2^3) \quad \not\cong \quad k[\FF_3^2]
$$
as follows.
The action of $G=GL_2(\FF_3)$ commutes with the action of its 
center $C=\{\pm I_{2 \times 2}\} \cong \ZZ/2\ZZ$.  Thus a $kG$-module 
isomorphism $Q_k\cong k[\FF_3^2]$ would necessarily 
lead to a $k[G \times C]$-module isomorphism, and
hence also $kG$-module isomorphisms between the $C$-isotypic subspaces
$Q_k^-$ and $k[\FF_3^2]^-$
where for $U=Q_k$ or $k[\FF_3^2]$ we define
\[
U^-:= \{ u \in U\text{ such that } -I_{2 \times 2} : u \mapsto -u\}.
\]
It therefore suffices to check that these two isotypic
subspaces are {\it not} $kG$-module isomorphic:
$$
\begin{aligned}
Q_k^{-} &= \{f \in Q_k: f(-x_1,-x_2) = -f(x_1,x_2) \} 
=(Q_k)_1 \oplus (Q_k)_3\\
 &=\spanof_{k}\{\quad x_1,\quad x_2 \quad \} \quad \oplus  \quad 
   \spanof_{k}\{\quad x_1^2 x_2, \quad x_1 x_2^2\quad \},\\
 & \\
k[\FF_3^2]^{-}&=\{ w \in k[\FF_3^2]: -I_{2 \times 2}: w \mapsto -w\} \\
        &=\spanof_{k}\left\{w_1:=e_{(+1,0)}-e_{(-1,0)}, \right. \\
         &  \qquad  \qquad  \quad  w_2:=e_{(0,+1)}-e_{(0,-1)},  \\
          &   \qquad   \qquad \quad w_3:=e_{(+1,+1)}-e_{(-1,-1)}, \\
            &   \qquad \qquad \quad \left.  w_4:=e_{(-1,+1)}-e_{(+1,-1)} \right\}.
\end{aligned}
$$
To argue that $Q_k^{-} \not\cong k[\FF_3^2]^{-}$ as $kG$-modules, check that this
transvection in $G$
$$
u=\left[ 
\begin{matrix}
1 & 1\\
0 & 1\\
\end{matrix}
\right]
$$
acts on both of the $2$-dimensional summands $(Q_k)_1$ and $(Q_k)_3$ 
of $Q_k^{-}$ via $2 \times 2$ Jordan blocks, but it
acts on the $4$-dimensional space $k[\FF_3^2]^{-}$ by fixing $w_1$ and cyclically 
permuting $w_2 \mapsto w_3 \mapsto w_4 \mapsto w_2$. This $3$-cycle action is
 conjugate to a $3 \times 3$ Jordan block when $k$ has characteristic $3$.  
\end{example}

Although $Q_k$ and $\ugpark{k}$ are not {\it isomorphic} as $kG$-modules,
they do turn out to be {\it Brauer isomorphic}.  This Brauer-isomorphism
was essentially observed by Kuhn \cite{Kuhn}; 
see Remark~\ref{Kuhn-remark} below.

\begin{definition}
Recall \cite[Chapter 18]{Serre} that two finite-dimensional representations $U_1, U_2$ of a finite group $G$ over a field $k$ are
said to be {\it Brauer-isomorphic} as $kG$-modules, written $U_1 \approx U_2$,
if each simple $kG$-module has the same 
composition multiplicity in $U_1$ as in $U_2$.
Equivalently, each {\it $p$-regular element} $g$ in $G$ has the same {\it Brauer character values}
$\chi_{U_1}(g) = \chi_{U_2}(g)$.
\end{definition}

\noindent
In fact, when the field extension $k$ of $\FF_q$ actually contains $\FF_{q^m}$,
it is useful to consider an extra cyclic group 
$$
C:=\FF_{q^m}^\times \cong \ZZ/(q^m-1)\ZZ
$$ 
acting on both $Q_k$ and $\ugpark{k}$ in a way that commutes with the $G$-actions.

\begin{definition}[$C$-action on the graded parking space]
\label{graded-C-action}
When $k \supset \FF_{q^m}$, an element $\gamma$ in $C=\FF_{q^m}^\times$
acts on $S_k=k[x_1,\ldots,x_n]$ by the scalar variable substitution
$$
x_i \longmapsto \gamma x_i\text{ for }i=1,2,\ldots,n.
$$
This $C$-action preserves $\mm^{[q^m]}=(x_1^{q^m},\ldots,x_n^{q^m})$, so 
that it descends to a $C$-action
on $Q_k$.  Also this $C$-action commutes with the action of $G$, so that
$Q_k$ becomes a $k[G \times C]$-module.  
\end{definition}

Note that the $C$-action on $Q_k$ depends in a trivial way on the {\it grading} structure
of $Q_k$:  an element $\gamma$ of $C=\FF_{q^m}^\times$ scales
all elements of a fixed degree $d$ in  
$Q_k$ by the same scalar $\gamma^d$.

\begin{definition}[$C$-action on the ungraded parking space]
When $k \supset \FF_{q^m}$,
an element $\gamma$ in $C=\FF_{q^m}^\times$
permutes the elements of $\FF_{q^m}^n$ via diagonal scalings:
$$
v=(v_1,\ldots,v_n) \overset{\gamma}{\longmapsto} (\gamma v_1,\ldots,\gamma v_n).
$$
Again this commutes with the permutation action of $G$ on $\FF_{q^m}^n$,
giving $\ugpark{k}$ the structure of a permutation $k[G \times C]$-module.
\end{definition}

To understand why the $k[G \times C]$-modules
$Q_k$ and $\ugpark{k}$ are Brauer-isomorphic, we introduce a third object:
an ungraded ring $R_k$ that turns out to be a thinly disguised version
of $\ugpark{k}$.

\begin{definition}
Define an ungraded quotient ring $R_k$ of $S_k = k[x_1, \ldots, x_n]$ by
$$
R_k:=S_k/\nn
\quad \text{ where } \quad 
\nn:=(x_1^{q^m}-x_1,\ldots,x_n^{q^m}-x_n).
$$
As $\nn$ is stable under the $G \times C$-action on $S_k$,
the quotient $R_k$ inherits the structure of 
a $k[G \times C]$-module.  
\end{definition}

\begin{proposition}
\label{identifying-function-ring-prop}
When $k \supset \FF_{q^m}$,
one has a $k[G \times C]$-module isomorphism
$
R_k \cong k^{\FF_{q^m}^n}
$,
where $k^{\FF_{q^m}^n}$ is the ring of all $k$-valued functions on the
finite set $\FF_{q^m}^n$ with pointwise addition and multiplication.
In particular, $R_k$ is $k[G \times C]$-module isomorphic to the
contragredient of $\ugpark{k}$, and hence to
 $\ugpark{k}$ itself.
\end{proposition}
\begin{proof}
The map $S_k \longrightarrow k^{\FF_{q^m}^n}$ that evaluates
a polynomial $f(x_1,\ldots,x_n)$ at the points of $\FF_{q^m}^n$ is
well-known to be a surjective ring homomorphism with kernel $\nn$ when $k \supset \FF_{q^m}$. This proves most of the assertions.
For the last assertion, note that permutation representations
are self-contragredient.
\end{proof}

It will turn out that $S_k$ is closely related to
$R_k$ via a filtration
\begin{equation}
\label{ring-filtration}
F_0 \subset F_1 \subset F_2 \subset \cdots \subset R_k
\end{equation}
where $F_i$ is the image within $R_k$ of
polynomials in $S_k$ of degree at most $i$.
Note that $F_i F_j \subset F_{i+j}$, allowing one to define the 
{\it associated graded ring} 
$$
\gr_F R_k:= F_0 \oplus F_1/F_0 \oplus F_2/F_1 \oplus F_3/F_2 \oplus \cdots
$$
with multiplication $F_i/F_{i-1} \times F_j/F_{j-1} \rightarrow F_{i+j}/F_{i+j-1}$
induced from $F_i \times F_j \rightarrow F_{i+j}$.  

\begin{proposition}
\label{identifying-gr-prop}
When $k \supset \FF_{q^m}$, one has 
a $G \times C$-equivariant isomorphism of graded rings
$
Q_k \cong \gr_F R_k.
$
\end{proposition}
\begin{proof}
Consider the $k$-algebra map $\varphi$ defined by
\begin{equation}
\label{map-into-gr}
\begin{array}{rcll}
S_k & \overset{\varphi}{\longrightarrow} & \gr_F R_k & \\
x_i & \longmapsto & \overline{x}_i &\in F_1/F_0.
\end{array}
\end{equation}
We claim $\varphi$ surjects:  $R_k$ is
generated as a $k$-algebra by the images of $x_1,\ldots,x_n$, 
so the multiplication map 
$$
\underbrace{F_1 \times \cdots \times F_1}_{i \text{ factors}} \rightarrow F_i
$$
is surjective, and hence likewise for the induced
multiplication map
$
F_1/F_0 \times \cdots \times F_1/F_0
\rightarrow F_i/F_{i-1}.
$

The relation $x_i^{q^m}=x_i$ that holds in $R_k$
shows that $\overline{x}_i^{q^m}=\overline{x}_i=0$ inside
the $q^m$-graded component $F_{q^m}/F_{q^m-1}$ of $\gr_F R_k$.
Hence the surjection 
$S_k  \overset{\varphi}{\twoheadrightarrow}  \gr_F R_k$ 
has $\mm^{[q^m]}$ in its kernel, and
descends to a surjection 
$Q_k  \overset{\varphi}{\twoheadrightarrow}  \gr_F R_k$.
But all of $Q_k$, $R_k$, $\gr_F R_k$  have dimension
$(q^m)^n$, so $\varphi$ is an isomorphism.
Furthermore, it is easily seen to be $G \times C$-equivariant.
\end{proof}

\begin{corollary}
\label{Brauer-isomorphism-corollary}
When $k \supset \FF_q$, one
has a Brauer-isomorphism of $kG$-modules
$
Q_k 
\approx
\ugpark{k}.
$
Furthemore, when $k \supset \FF_{q^m}$ then
it is a Brauer-isomorphism of $k[G \times C]$-modules.
\end{corollary}
\begin{proof}
One may assume without loss of generality that
$k \supset \FF_{q^m}$, as one has Brauer-isomorphisms between
two $kG$-modules if and only if the same holds after extending scalars
to any field containing $k$.  

Then one has a string of $k[G \times C]$-module Brauer isomorphisms and
isomorphisms
$$
\ugpark{k} 
\cong
R_k  \approx 
\gr_F R_k 
\cong Q_k 
$$
derived, respectively, 
from Proposition~\ref{identifying-function-ring-prop},
from the filtration defining $\gr_F R_k$, and 
from Proposition~\ref{identifying-gr-prop}.
\end{proof}

\begin{remark}
\label{Kuhn-remark}
The Brauer isomorphism of $kG$-modules asserted 
in Corollary~\ref{Brauer-isomorphism-corollary}, ignoring the $C$-action,
is essentially a result of N. Kuhn, as we now explain.

Note that a choice of $\FF_q$-vector space basis for $\FF_{q^m}$ 
identifies $\FF_{q^m}$ with the length $m$ row vectors
$\FF_q^m$, and hence also identifies
$\FF_{q^m}^n$ with the $n \times m$ matrices $\FF_q^{n \times m}$.
Hence the $kG$-module $k[\FF_{q^m}^n]$ is isomorphic to the
the permutation action of $G=GL_n(\FF_q)$ left-multiplying matrices in
$\FF_q^{n \times m}$.  

Kuhn proved \cite[Theorem 1.8]{Kuhn} via similar filtration methods to ours that, for $k=\FF_p$ with $p$ a prime, the quotient
ring $Q_k:=k[x_1,\ldots,x_n]/(x_1^{p^m},\ldots,x_n^{p^m})$ 
has the same composition factors as the permutation representation 
$k[\FF_p^{n \times m}]$ on the space of $n \times m$ matrices.  In fact, he
proves this holds not only as $kG$-modules for $G=GL_n(\FF_p)$, 
but even as modules
over the larger semigroup ring $k[ \Mat_n(\FF_p) ]$ of $n \times n$ matrices,
which still acts by linear substitutions on $Q_{k}$ and
acts by matrix left-multiplication on $k[\FF_p^{n \times m}]$.
\end{remark}

\begin{remark}
The authors thank N. Kuhn for pointing out the following consequence
of Conjecture~\ref{mod-powers-conjecture}.
Since the filtration $F=\{F_i\}$ on the ring $R:=R_k$ defined 
in \eqref{ring-filtration} is $G$-stable, it induces a filtration
$F^G=\{(F_i)^G\}$ on the $G$-fixed subring $R^G$.
One has well-defined injective maps 
$(F_i)^G/(F_{i-1})^G \hookrightarrow F_i/F_{i-1}$,
whose images lie in the subspace $(F_i/F_{i-1})^G$.  Compiling these injections
gives an injective ring homomorphism
\begin{equation}
\label{Kuhns-injection}
\gr_{F^G}(R^G) \hookrightarrow \left( \gr_F R\right)^G.
\end{equation}

\begin{proposition}
\label{Kuhn-conjecture}
The specialization of 
Conjecture~\ref{mod-powers-conjecture} to $t=1$
is equivalent to the injection \eqref{Kuhns-injection}
being an isomorphism.  
\end{proposition}

\noindent
Thus Conjecture~\ref{mod-powers-conjecture} implies that
the operations of taking 
$G$-fixed points and forming the associated graded ring commute
when applied to the ungraded parking space $R$.

\begin{proof}[Proof of Proposition~\ref{Kuhn-conjecture}]
Proposition~\ref{identifying-gr-prop} shows that
$Q \cong \gr R$, and hence $Q^G \cong (\gr R)^G$.
Thus the specialization of Conjecture~\ref{mod-powers-conjecture} to $t=1$
is equivalent to the assertion that 
$$
\sum_{k=0}^{\min(m,n)} \qbin{m}{k}{q}=  \dim_{\FF_q} Q^G = \dim_{\FF_q} (\gr R)^G. 
$$  
On the other hand, Theorem~\ref{orbits-identified-theorem} below 
shows that the sum on the left equals $\dim_{\FF_q} R^G$, 
and hence also equals $\dim_{\FF_q}  \gr_{F^G}( R^G )$.  Thus Conjecture~\ref{mod-powers-conjecture} at $t=1$ asserts that the source and target
of the injection \eqref{Kuhns-injection} have the same dimension.
\end{proof}

\end{remark}

\subsection{$P_\alpha$-fixed spaces, orbits, 
and Parabolic Conjecture~\ref{mod-powers-conjecture}}

We next compare the $P_\alpha$-fixed spaces in
$Q_k$ and in $\ugpark{k}$.  Since $\ugpark{k}$ is a permutation
representation, one can identify its fixed space as
$$
\ugpark{k}^{P_\alpha} \cong k[ P_\alpha\backslash \FF_{q^m}^n ]
$$
where $P_\alpha\backslash \FF_{q^m}^n$ is the set of $P_\alpha$-orbits
on $\FF_{q^m}^n$.  This orbit set 
$P_\alpha\backslash \FF_{q^m}^n$
turns out to be closely related to the mysterious summation in the definition
\eqref{parabolic-Catalan-with-m-definition} of $C_{\alpha,m}(t)$.

\begin{definition}
Let $\beta=(\beta_1,\ldots,\beta_\ell)$ be a weak composition
having $|\beta| \leq m$, and define its
partial sums $B_i=\beta_1+\beta_2+\cdots+\beta_i$ as usual.
A {\it $(\beta,m-|\beta|)$-flag} in $\FF_{q^m}$ is a tower 
\begin{equation}
\label{typical-flag}
0=V_{B_0} \subset V_{B_1} \subset V_{B_2} \subset \cdots
\subset V_{B_\ell} \subset \FF_{q^m}
\end{equation}
of $\FF_q$-subspaces inside $\FF_{q^m}$
with $\dim_{\FF_q} V_{B_i} = B_i$ for each $i$.  
\end{definition}

Let $Y_\beta$ be the
set of  $(\beta,m-|\beta|)$-flags in $\FF_{q^m}$,
whose cardinality is known to be a {\it $q$-multinomial coefficient}
$$
|Y_\beta|=\qbin{m}{\beta,m-|\beta|}{q}
:=\qbin{m}{\beta,m-|\beta|}{q,t=1}
= \frac{ \prod_{j=0}^{n-1} ( q^n-q^j ) }
       { \prod_{i=1}^\ell \prod_{j=0}^{\beta_i-1} ( q^{B_i}-q^{B_{i-1}+j} ) }.
$$
Given a composition $\alpha$ of $n$, define the set 
$$
X_\alpha
:=\bigsqcup_{\substack{ \beta: 
\beta \leq \alpha,\\ |\beta| \leq m}} Y_{\beta},
$$
which has cardinality given by
$$
|X_\alpha|
=\sum_{\substack{ \beta: 
\beta \leq \alpha,\\ |\beta| \leq m}} |Y_{\beta}|
= \left[ C_{\alpha,m}(t) \right]_{t=1}.
$$

\begin{theorem}
\label{orbits-identified-theorem}
The set $X_\alpha$ naturally indexes $P_\alpha\backslash \FF_{q^m}^n$.
Therefore  
$$
\dim_k \ugpark{k}^{P_\alpha}= | P_\alpha \backslash \FF_{q^m}^n |=\left[ C_{\alpha,m}(t) \right]_{t=1}.
$$
\end{theorem}
\begin{proof}
Fix $\alpha=(\alpha_1,\ldots,\alpha_\ell)$ and 
denote its partial sums by
$A_i=\alpha_1+\alpha_2+\cdots+\alpha_i$ as usual.
To any vector 
$v=(v_1,\ldots,v_n)$ in $\FF_{q^m}^n$ one can associate
a flag $(V_i)_{i=1}^\ell$ in $\FF_{q^m}$ defined by
$V_i:=\spanof_{\FF_q}\{ v_1,v_2,\ldots,v_{A_i} \}$.
This gives rise to a weak composition 
$\beta=(\beta_1,\ldots,\beta_\ell)$
with
$$
\beta_i =\dim_{\FF_q} V_i - \dim_{\FF_q} V_{i-1} 
=\dim_{\FF_q} V_i/V_{i-1} \leq \alpha_i,
$$
where the inequality arises because $V_i/V_{i-1}$ is
spanned by the  $\alpha_i$ vectors
$\{ v_{A_{i-1}+1}, v_{A_{i-1}+2}, \ldots, v_{A_{i}}\}$.
Also one has
$$
|\beta|=\dim_{\FF_q} \spanof_{\FF_q}\{ v_1,v_2,\ldots,v_n \} \leq 
\dim_{\FF_q} \FF_{q^m}=m.
$$
Thus the flag $(V_i)_{i=1}^{\ell}$ associated to $v$ lies in $Y_\beta \subset X_\alpha$,
and this flag is a complete invariant
of the $P_\alpha$-orbit of $v$:
one has $P_\alpha v = P_\alpha v'$ 
if and only if $V_i = V'_i$ for $i=1,2,\ldots,\ell$.
This gives  a bijection 
$P_\alpha \backslash \FF_{q^m}^n \rightarrow X_\alpha$.  
\end{proof}

\begin{remark}
When $\alpha=(n)$ so that $P_\alpha=G=GL_n(\FF_q)$, the analysis of the $G$-orbits $G\backslash \FF_{q^m}^n$ just given in Theorem~\ref{orbits-identified-theorem} is closely related to Kuhn's analysis in \cite[\S 5, Cor. 5.3]{Kuhn};  see also Remark~\ref{Kuhn-remark} above.
\end{remark}

Something even more striking is true, regarding the action
of the cyclic group $C=\FF_{q^m}^\times$ 
on the set of flags $X_\alpha$ inside $\FF_{q^m}$.
Fix a multiplicative generator $\gamma$ for 
$C=\langle \gamma \rangle =\FF_{q^m}^\times$, so $\gamma$
has multiplicative order $q^m-1$.  Also fix a primitive
$(q^m-1)$st root of unity $\zeta$ in $\CC^\times$.
For an element $\gamma^d$ in $C$,
denote its fixed subset by
$$
\left( X_\alpha\right) ^{\gamma^d}:=\{ x \in X_\alpha: \gamma^d(x) = x \}.
$$

\begin{proposition}
\label{CSP-proposition}
For any composition $\alpha$ and integer $d$, one has
has 
$$
|\left( X_\alpha\right)^{\gamma^d}|=
\left[ C_{\alpha,m}(t) \right]_{t=\zeta^d}.
$$
In other words, the triple $(X_\alpha, C_{\alpha,m}(t), C)$
exhibits a cyclic sieving phenomenon in the sense of \cite{StantonWhiteR}.
\end{proposition}
\begin{proof}
It follows from \cite[Theorem 9.4]{StantonWhiteR}
that for a weak composition $\beta$ with $|\beta| \leq m$
and integer $d$, 
one has 
$$
|\left( Y_\beta\right)^{\gamma^d}| = \qbin{m}{\beta,m-|\beta|}{q,t=\zeta^d}.
$$
So by \eqref{parabolic-Catalan-with-m-definition} suffices to show $\left( Y_\beta\right)^{\gamma^d} \neq \varnothing$ implies
that 
$\left[ t^{e(m,\alpha,\beta)} \right]_{t=\zeta^d}=1$.
  
One checks this as follows.   Let $r$ be the multiplicative
order of $\gamma^d$ within $C=\FF_{q^m}^\times$, and of 
$\zeta^d$ within $\CC^\times$.  
One knows that $\FF_q(\gamma^d) =\FF_{q^{\ell}}$
for some divisor $\ell$ of $m$ with the property that $r$ divides $q^{\ell}-1$.
Then any $(\beta,m-|\beta|)$-flag of $\FF_q$-subspaces in $\FF_{q^m}$
stabilized by $\gamma^d$ must actually be a flag
of $\FF_q(\gamma^d)$-subspaces, and hence a flag of 
$\FF_{q^{\ell}}$-subspaces.  Therefore $\ell$ must divide each
partial sum $B_i$ for $i=1,2,\ldots,\ell$.  As $\ell$ also divides $m$,
this means that $\ell$ divides each $m-B_i$, so that
$q^{\ell}-1$ divides each $q^{m-B_i}-1$, and hence
$q^{\ell}-1$ divides each $q^m-q^{B_i}=q^{B_i}(q^{m-B_i}-1)$.
This means that $r$ will also divide each $q^m-q^{B_i}$, so
that $r$ divides $e(m,\alpha,\beta)$, and 
$\left[ t^{e(m,\alpha,\beta)} \right]_{t=\zeta^d}=1$
as desired.
\end{proof}

One can reinterpret Proposition~\ref{CSP-proposition} in the following
fashion.

\begin{reinterp}
Parabolic Conjecture~\ref{mod-powers-conjecture}
implies that for any field $k \supset \FF_{q^m}$, one has
a $kC$-module isomorphism of the $P_\alpha$-fixed spaces
\begin{equation}
\label{CSP-proposition-isomorphism}
Q_k^{P_\alpha} \cong \ugpark{k}^{P_\alpha}.
\end{equation}
\end{reinterp}
\begin{proof}
Note that $|C|=q^m-1$ is relatively prime to the characteristic of $k \supset \FF_q$,
and hence $kC$ is semisimple.
Thus it suffices to check that 
$Q_k^{P_\alpha}$ and  $\ugpark{k}^{P_\alpha}$
have the same $kC$-module Brauer characters. 
Recall \cite[\S 18.1]{Serre} that to compute these Brauer characters, 
one starts by fixing an embedding of cyclic groups 
$$
\begin{array}{rcll}
C=\FF_{q^ m}^\times=\langle \gamma \rangle & \longrightarrow &  \CC^\times & \\
                            \gamma^d & \longmapsto & \zeta^d&\text{ where }\zeta := e^{\frac{2 \pi i}{q^m-1}}. 
\end{array}
$$
Then whenever an element $\gamma^d$ in $C$ acts in some 
$r$-dimensional $\FF_{q^m} C$-module $U$ with multiset of
eigenvalues $(\gamma^{i_1},\ldots,\gamma^{i_r})$, its Brauer character
value on $U$ is defined to be
$$
\chi_U(\gamma^d) := \zeta^{i_1}+ \cdots +\zeta^{i_r}.
$$

To compute Brauer character values 
on $Q_k^{P_\alpha}$, recall from Definition~\ref{graded-C-action} that the
element $\gamma^d$ in $C$ acting on this graded vector space 
will scale the $e$th homogeneous component by $(\gamma^d)^e$.
Hence 
\begin{equation}
\label{graded-Brauer-character}
\chi_{Q_k^{P_\alpha} }(\gamma^d) =
\left[ \Hilb\left(  Q_k^{P_\alpha}, t \right) \right]_{t=\zeta^d}. 
\end{equation}

To compute the Brauer character values on
$\ugpark{k}^{P_\alpha}$, note that since  $\ugpark{k}$ is a 
permutation representation of $P_\alpha \times C$, 
its $P_\alpha$-fixed space $\ugpark{k}^{P_\alpha}$ is isomorphic to
the permutation representation of $C$ on the set of $P_\alpha$-orbits 
on $P_\alpha \backslash \FF_{q^m}^n$.  Equivalently, by 
Theorem~\ref{orbits-identified-theorem},
this is the permutation representation of $C$ on $X_\alpha$.
For a {\it permutation} representation of a finite group, 
it is easily seen that its Brauer character value for a ($p$-regular) element is its usual ordinary complex character value, that is, its
number of fixed points.  Hence the Brauer character value for $\gamma^d$ when acting on $\ugpark{k}^{P_\alpha}$ is $| \left( X_\alpha\right)^{\gamma^d} |$.
Comparing this value with \eqref{graded-Brauer-character},
and assuming Parabolic Conjecture~\ref{mod-powers-conjecture},
one finds that Proposition
\ref{CSP-proposition} exactly asserts that the two $kC$-modules in
\eqref{CSP-proposition-isomorphism} have the same Brauer characters.
\end{proof}

\section{Further questions and remarks}
\label{questions-and-remark-section}

\subsection{The two limits where $t,q$ go to $1$.}
In  \cite[(1.3)]{StantonR}, it was noted that two different kinds of limits 
applied to the $(q,t)$-binomials yield the same answer after swapping $q$ and $t$, namely
$$
\lim_{ t \rightarrow 1 } \qbin{n}{k}{q,t}
 =\qbin{n}{k}{q} 
 \qquad \text{ and } \qquad
\lim_{ q\rightarrow 1 }  \qbin{n}{k}{q,t^{\frac{1}{q-1}} }
  =\qbin{n}{k}{t}. 
$$
One can similarly apply these two kinds of limits to $C_{n,m}(t)$, giving two somewhat different answers:
\begin{eqnarray}
\label{first-limit}
\lim_{ t \rightarrow 1 } C_{n,m}(t)
  &=&\sum_{k=0}^{\min(n,m)}
       \qbin{m}{k}{q}  \\
\label{second-limit}
 \lim_{ q\rightarrow 1 } C_{n,m}(t^{\frac{1}{q-1}})
  &=&\sum_{k=0}^{\min(n,m)}
         t^{(n-k)(m-k)} \qbin{m}{k}{t}. 
\end{eqnarray}
The limit \eqref{first-limit} can be interpreted, via
Theorem~\ref{orbits-identified-theorem} for $\alpha=(n)$, as
counting $GL_n(\FF_q)$-orbits on $\FF_{q^m}^n$.
When $m \geq n$, it gives the {\it Galois number} $G_n$ 
counting all $\FF_q$-subspaces of $\FF_q^n$ and studied, e.g., 
by Goldman and Rota \cite{GoldmanRota}.   
We have no insightful 
explanation or interpretation for the limit \eqref{second-limit}.

In addition, it is perhaps worth noting two 
further specializations of \eqref{second-limit}:  
setting $m=n$ or $m=n-1$, and then taking the 
limit as $n \rightarrow \infty$, one
obtains the left sides of the two {\it Rogers-Ramanujan identities}:
$$
\sum_{k=0}^\infty \frac{t^{k^2}}{(t;t)_k} 
=\frac{1}{(t;t^5)_\infty(t^4;t^5)_\infty}   
\qquad \text{ and } \qquad
\sum_{k=0}^\infty \frac{t^{k^2+k}}{(t;t)_k}
=\frac{1}{(t^2;t^5)_\infty(t^3;t^5)_\infty}   
$$
where $(x;t)_k:=(1-x)(1-tx)\cdots(1-t^{k-1}x)$ and 
$(x;t)_\infty=\lim_{k \rightarrow \infty} (x;t)_k$.
We have no explanation for this.

\subsection{$G$-fixed divided powers versus $G$-cofixed polynomials}
We reformulate Conjecture~\ref{cofixed-conjecture} slightly.

Setting $V:=\FF_q^n$, one can regard the symmetric algebra
$S=\FF_q[x_1,\ldots,x_n]=\Sym(V^*)$  as a {\it Hopf algebra}, which is graded of {\it finite type}, meaning that each graded piece
$S_d$ is finite-dimensional.  
Then the {\it (restricted) dual} Hopf algebra $D(V)$ 
has as its $d$th graded piece $D(V)_d= S_d^*$, the 
$\FF_q$-dual vector space to $S_d$, and naturally
carries the structure of a {\it divided power algebra} on $V$;  
see, e.g., \cite[\S I.3,I.4]{AkinBuchsbaumWeyman}.
Consequently, Proposition~\ref{fixed-cofixed-duality} implies that 
the $G$-fixed space $D(V)_d^G $ is $\FF_q$-dual to the $G$-cofixed space $(S_d)_G$, so that
$$
\Hilb(D(V)^G, t) = \Hilb(S_G,t).
$$
This means one can regard Conjecture~\ref{cofixed-conjecture} as being about $\Hilb(D(V)^G, t)$ instead.
Since $D(V)^G$ is a subalgebra of the divided power algebra $D(V)$, this suggests the following.

\begin{question}
For $V=\FF_q^n$ and $G=GL_n(\FF_q)$,  
is Conjecture~\ref{cofixed-conjecture} suggesting a predictable or well-behaved ring structure for the $G$-fixed subalgebra $D(V)^G$ of the divided power algebra $D(V)$?
\end{question}

The invariant theory literature  for finite subgroups of $GL(V)$ acting on divided powers $D(V)$ is much less
extensive than the literature for actions on polynomial rings $S=\Sym(V)$,
although one finds a few results in Segal \cite{Segal}.    M. Crabb\footnote{Personal communication, 2013.} informs us that, in work with J. Hubbuck and a student D. Salisbury, 
some results on the structure of $D(V)^G$ were known to them
for $G=GL_2(\FF_p)$ acting on $V=\FF_p^2$ with $p=2,3$. 

\subsection{Homotopy theory}
The paper of Kuhn \cite{Kuhn} mentioned in Section~\ref{CSP-section} is part of a large literature relating modular representations of $GL_n(\FF_q)$ and its action on
$S=\FF_q[x_1,\ldots,x_n]$ to questions about stable splittings in homotopy theory.  In this work, an important role is played by a commuting action on $S$ of the mod $p$ Steenrod algebra;  some references are
Smith \cite[Chapters 10,11]{Smith},  
Carlisle and Kuhn \cite{CarlisleKuhn}, two surveys 
by Wood \cite{Wood1} and \cite[\S 7]{Wood2}, and the papers of
Doty and Walker \cite{DotyWalker1, DotyWalker2, DotyWalker3}.
We have not seen how to use these results in attacking Parabolic Conjectures~\ref{mod-powers-conjecture} and \ref{cofixed-conjecture}.

\subsection{Approaches to Conjecture~\ref{mod-powers-conjecture}}

In approaching Conjecture~\ref{mod-powers-conjecture} we would like an 
explicit $\FF_q$-basis for $Q^G$ where 
$Q=S/ \mm^{[q^m]}$, in degrees suggested by
the $(q,t)$-binomial summands in the formula
\eqref{Catalan-with-m-definition} for $C_{n,m}(t)$.
For example, when $m \geq n$ one can at least make 
a reasonable guess about {\it part} of such a basis that
models the $k=n$ summand in \eqref{Catalan-with-m-definition},
as follows.  It was shown in \cite[(5.6)]{StantonR} that
$$
\qbin{m}{n}{q,t}
 = \sum_{(\lambda,a)} t^{\sum_{i=0}^{n-1}a_i(q^n-q^{n-i})}
$$
where $(\lambda,a)$ ranges over all pairs in which
$\lambda=(\lambda_1,\ldots,\lambda_n)$ satisfies
$m-n \geq \lambda_1 \geq \cdots \geq \lambda_n \geq 0$,
and $a=(a_0,\ldots,a_{n-1})$ is a tuple
of nonnegative integers {\it $q$-compatible} with $\lambda$ in the
sense that 
$a_i \in [\delta_i, \delta_i + q^{\lambda_i})$, where
$
\delta_i := q^{\lambda_{i+1}} +  q^{\lambda_{i+1}+1} + \cdots 
                          + q^{\lambda_i-1}. 
$
Thus one might guess that the images of the monomials 
$\prod_{i=0}^{n-1} D^{a_i}_{n,n-i}$
as one ranges over the same pairs of $(\lambda,a)$ form
part of an $\FF_q$-basis for $Q^G$, and their $\FF_q$-linear independence
has been checked computationally for a few small values of $n, m, q$.

However, one knows that at least {\it some} of the basis elements accounting for
other summands in \eqref{Catalan-with-m-definition} 
are {\it not} sums of products of Dickson polynomials $D_{n,i}$, as the 
natural map $S^G \rightarrow Q^G$ is {\it not} surjective
for $n \geq 2$.  One seems to need recursive constructions, that produce invariants in $n$ variables from invariants in $n-1$ variables, with predictable effects 
on the degrees.  Currently, we lack such constructions.

Non-surjectivity of $S^G \rightarrow Q^G$ appears
in another initially promising approach.
As $\mm^{[q^m]}=(x_1^{q^m}, \ldots,x_n^{q^m})$ is generated by a 
regular sequence on $S$, one has an $S$-free
{\it Koszul resolution} \cite[\S XVI.10]{Lang} 
for $Q=S/\mm^{[q^m]}$:
$$
0 \rightarrow S \otimes_\FF \wedge^n V \rightarrow
\cdots \rightarrow S \otimes_\FF \wedge^2 V \rightarrow
S \otimes_\FF \wedge^1 V \rightarrow
S \rightarrow Q \rightarrow 0.
$$
Taking $G$-fixed spaces gives a {\it complex}, which is generally not
exact when $\FF_q G$ is not semisimple,
but at least contains $Q^G$ at its right end:
\begin{equation}
\label{non-exact-complex}
0 \rightarrow (S \otimes_\FF \wedge^n V)^G \rightarrow
\cdots \rightarrow (S \otimes_\FF \wedge^2V)^G \rightarrow
(S \otimes_\FF \wedge^1 V)^G \rightarrow
S^G \rightarrow Q^G \rightarrow 0.
\end{equation}
A result of Hartmann and Shepler \cite[\S 6.2]{HartmannShepler}
very precisely describes each term
$(S \otimes_\FF \wedge^i V)^G$ in  \eqref{non-exact-complex} as a
free $S^G$-module with explicit $S^G$-basis elements
that are homogeneous with predictable degrees;  this is an
analogue of a classic result on invariant differential forms
for complex reflection groups due to
Solomon \cite{Solomon}.
Thus each term $(S \otimes_\FF \wedge^i V)^G$
has a simple explicit Hilbert series.  However, non-exactness
means that \eqref{non-exact-complex} is not a resolution of $Q^G$, 
so it does not let us directly compute its Hilbert series.

\subsection{Rational Cherednik algebras for $GL_n(\FF_q)$.}
\label{RCA-question}
Section~\ref{CSP-section} alluded to the
considerations that led to Conjecture~\ref{mod-powers-conjecture}, coming
from the theory of real reflection groups $W$.  When $W$ acts
irreducibly on $\RR^n$ and on  
the polynomial algebra $\CC[\xx]=\CC[x_1,\ldots,x_n]$, one
can define its graded {\it $W$-parking space} 
$\CC[\xx]/(\theta_1,\ldots,\theta_n)$, as a quotient by a 
certain homogeneous system of parameters $\theta_1,\ldots,\theta_n$ of degree $h+1$ inside $\CC[\xx]$, 
where $h$ is the {\it Coxeter number} of $W$; see 
\cite{ArmstrongRhoadesR}.  

Replacing $W$ by $G:=GL_n(\FF_q)$, we
think of $h:=q^n-1$ as the {\it Coxeter number}, with $x_i^{q^n}$ playing the role of $\theta_i$, and $Q=S/\mm^{q^n}$ playing the role of the 
graded $G$-parking space.

In the real reflection group theory, the $W$-parking space 
carries the structure of an irreducible finite dimensional representation $L_c(\triv)$ for the {\it rational Cherednik algebra} $H_c(W)$ with parameter value $c=\frac{h+1}{h}$.  Here the $\theta_i$ span the common kernel of the
{\it Dunkl operators} in $H_c(W)$ when acting on $\CC[\xx]=M_c(\triv)$.  In addition, its $W$-fixed space $L_c(\triv)^W$ is a graded subspace whose Hilbert series is the {\it $W$-Catalan polynomial}.  

This explains why we examined the Hilbert series of $Q^G$ in our context.
In fact, rational Cherednik algebras $H_c(G)$ for $G=GL_n(\FF_q)$ and 
their finite dimensional representations $L_c(\triv)$ have been studied by
Balagovi\'c and Chen \cite{BalagovicChen}.  However, their results show that 
the common kernel of the Dunkl operators in $H_c(G)$ acting on $S=\FF_q[\xx]$ 
is {\it not} spanned by $x_1^{q^n},\ldots,x_n^{q^n}$.  In fact, for almost all choices 
of $n$ and the prime power $q=p^r$, 
they show \cite[Theorem 4.10]{BalagovicChen} that it is 
spanned by $x_1^p,\ldots,x_n^p$, independent of the exponent $r$.
  
Can one modify this rational Cherednik theory for $G$ to better fit our setting, and gain insight into $Q^G$?

\appendix

\section{Proof of Proposition~\ref{equality-up-to-q^m-prop}}
\label{proof-of-equality-up-to-q^m-prop-section}

We recall here the statement to be proven.

\vskip.1in
\noindent
{\bf Proposition \ref{equality-up-to-q^m-prop}.}
{\it 
For any $m \geq 0$ and any composition $\alpha$ of $n$, the power
series 
\[
\Hilb(S^{P_\alpha},t)
=
\prod_{i=1}^\ell \prod_{j=0}^{\alpha_i-1}
\frac{1}{1 - t^{q^{A_i}-q^{A_{i-1}+j}}}
\]
is congruent in $\ZZ[[t]]/(t^{q^m})$ to the polynomial
\[
C_{\alpha,m}(t) 
=
\displaystyle \sum_{\substack{\beta: 
\beta \leq \alpha\\ |\beta| \leq m}}
 t^{e(m,\alpha,\beta)}
 \qbin{m}{\beta,m-|\beta|}{q,t}
\qquad \text{ where } \qquad
e(m,\alpha,\beta) =\sum_{i=1}^\ell  (\alpha_i-\beta_i) (q^m-q^{B_i}).
\]
}
\vskip.1in

\noindent
Fix $m \geq 0$.  
Throughout this proof, ``$\equiv$'' denotes
equivalence in $\ZZ[[t]]/(t^{q^m})$.

Given the composition $\alpha=(\alpha_1,\ldots,\alpha_\ell)$,
denote its $i$th partial sum by $A_i=\alpha_1+\alpha_2+\cdots+\alpha_i$
as before.  Adopting the convention that $A_0:=0, A_{\ell+1}:=+\infty$,
define $L$ to be the largest index in $0 \leq L \leq \ell$ for which
$A_L \leq m$, so that $A_{L+1} > m$.  
Part of the relevance of the index $L$ comes from the truncation
to the first $L$ factors in the product formula 
\begin{equation}
\label{parabolic-hilb-final-form}
\Hilb(S^{P_\alpha},t) 
=\prod_{i=1}^\ell \prod_{j=0}^{\alpha_i-1}
\left(
1 - t^{q^{A_i}-q^{A_{i-1}+j}}
\right)^{-1}
\equiv \prod_{i=1}^L \prod_{j=0}^{\alpha_i-1}
\left(
1 - t^{q^{A_i}-q^{A_{i-1}+j}}
\right)^{-1},
\end{equation}
where the last equivalence is justified as follows.
As $q$ is a prime power, one has $q \geq 2$.  Thus for integers $a, b, c$, one has
\begin{equation}
\label{truncation-inequality}
a > b, c \quad \text{ implies } \quad
q^a - q^b - q^c \geq q^a - 2q^{a-1} = (q-2) q^{a-1} \geq 0.
\end{equation}
In particular, 
$q^{A_i}-q^{A_{i-1}+j} \geq q^{A_i-1} \geq q^m$ for all $i \geq L+1$.   Thus, all of the factors in \eqref{parabolic-hilb-final-form} with $i > L$ are equivalent to $1$ modulo $(t^{q^m})$.

We will make frequent use of \eqref{truncation-inequality}; for example, it
helps prove the following
lemma, which shows that most summands of $C_{\alpha,m}(t)$
in \eqref{parabolic-Catalan-with-m-definition}
vanish in $\ZZ[[t]]/(t^{q^m})$.

\begin{lemma}
\label{vanishing-t-power-lemma}
Given $m$ and $\alpha$, with $A_i$ and $L$ defined as above, the 
weak compositions $\beta=(\beta_1,\ldots,\beta_\ell)$ with 
$0 \leq \beta \leq \alpha$ and $|\beta| \leq m$ for which
$e(m,\alpha,\beta) < q^m$ are exactly those of the
following two forms: 
\begin{enumerate}
\item[(i)] either
$\displaystyle
\beta= \widehat{\alpha} :=
\begin{cases}
\alpha & \textrm{ if } L = \ell, \\
(\alpha_1,\ldots,\alpha_L,m-A_L,0,\ldots,0) & \textrm{otherwise},
\end{cases}
$
\item[(ii)] 
or for $k=1,2,\ldots,L$,
\[
\beta = \widehat{\alpha}^{(k)}:=
\begin{cases}
(\alpha_1,\ldots,\alpha_{k-1}, \alpha_k - 1,
    \alpha_{k+1},\ldots,\alpha_\ell) & \textrm{ if } L = \ell, \\
(\alpha_1,\ldots,\alpha_{k-1}, \alpha_k - 1,
    \alpha_{k+1},\ldots,\alpha_L,m-A_L+1,0,\ldots,0) & \textrm{otherwise}.
\end{cases}
\]
\end{enumerate}
In the former case, $e(m,\alpha,\beta)=0$, and in the latter, $e(m,\alpha,\beta)=q^m-q^{A_k-1}$.
\end{lemma}

\begin{proof}[Proof of Lemma~\ref{vanishing-t-power-lemma}]
Assume $\beta=(\beta_1,\ldots,\beta_\ell)$ has $0 \leq \beta \leq \alpha$,
with $|\beta| \leq m$,  and that $e(m,\alpha,\beta) < q^m$.  As before, let $B_i=\beta_1+\beta_2+\cdots+\beta_i$
for $i=0,1,\ldots,\ell+1,$ with conventions $B_0:=0$ and $B_{\ell+1}:=m$.  
By \eqref{truncation-inequality}, the condition $e(m, \alpha, \beta) < q^m$ implies that at most one summand in $e(m, \alpha, \beta)$ may be nonzero, and if the $i$th summand $(\alpha_i - \beta_i)(q^m - q^{B_i})$ is nonzero then $\alpha_i - \beta_i = 1$.  Choose $j$ minimal so that $0 \leq j \leq \ell + 1$ and $B_j = m$.  We consider two cases, depending on whether or not $e(m, \alpha, \beta) = 0$.
\vskip.1in
\noindent
{\sf Case 1.} $e(m, \alpha, \beta) = 0$.  In this case all summands in $e(m, \alpha, \beta)$ are zero, so $\beta_i = \alpha_i$ for all $i < j$.  If $j = \ell + 1$ then it follows immediately that $\beta = \alpha = \widehat{\alpha}$.  Otherwise, $j \leq \ell$.  Since $B_j = m$ but $A_i = B_i < m$ for $i < j$, we have $j = L$.  Therefore $\beta = (\alpha_1,\ldots,\alpha_L,m-A_L,0,\ldots,0) = \widehat{\alpha}$ in this case.

\vskip.1in
\noindent
{\sf Case 2.} $e(m, \alpha, \beta) > 0$.  In this case there is an index $k$ such that $k < j$ and $\alpha_i - \beta_i = 1$, and for all other $i < j$ we have $\beta_i = \alpha_i$.  If $j = \ell + 1$ then it follows immediately that $\beta = (\alpha_1,\ldots,\alpha_{k-1}, \alpha_k - 1, \alpha_{k+1},\ldots,\alpha_\ell) = \widehat{\alpha}^{(k)}$.  Otherwise, $j \leq \ell$.  Since $B_j = m$ but $A_i \leq B_i + 1 \leq m$ for $i < j$, we have $j = L$.  Therefore $\beta = (\alpha_1,\ldots,\alpha_{k-1}, \alpha_k - 1,
    \alpha_{k+1},\ldots,\alpha_L,m-A_L+1,0,\ldots,0) = \widehat{\alpha}^{(k)}$ in this case.
\end{proof}

Returning to the proof of Proposition~\ref{equality-up-to-q^m-prop},
note that Lemma~\ref{vanishing-t-power-lemma} implies
\begin{equation}
\label{only-summands-left}
C_{\alpha,m}(t) \equiv \qbin{m}{\widehat{\alpha}}{q,t}
+ \sum_{k=1}^L t^{q^m-q^{A_k-1}} \qbin{m}{\widehat{\alpha}^{(k)}}{q,t}.
\end{equation}
We next process the summands on the right. By definition, one has that
\[
t^{q^m-q^{A_k-1}} \qbin{m}{\widehat{\alpha}^{(k)}}{q,t}
=
\left.
t^{q^m-q^{A_k-1}}
\prod_{j=0}^{A_L-1} (1-t^{q^m-q^j})
\middle/
\prod_{i=1}^L \prod_{j=0}^{\widehat{\alpha}^{(k)}_i-1}
(1 - t^{q^{\widehat{A}^{(k)}_i}-q^{\widehat{A}^{(k)}_{i-1}+j}})
\right.
,
\]
where here $\widehat{A}^{(k)}_i := \widehat{\alpha}^{(k)}_1 + \ldots + \widehat{\alpha}^{(k)}_i$ as usual.
We will attempt to simplify the fraction on the right side, 
working mod $(t^{q^m})$.
Note that in its numerator,
only $t^{q^m-q^{A_k-1}}$ survives,
as \eqref{truncation-inequality} implies
$
(q^m-q^{A_k-1})+(q^m-q^j) \geq q^m.
$ 
Meanwhile in its denominator,
only the factors indexed by $i=1,2,\ldots,k$ survive
multiplication by $t^{q^m-q^{A_k-1}}$
when working mod $(t^{q^m})$:
 since $\widehat{A}^{(k)}_i \geq A_k$ for $i \geq k+1$, 
\eqref{truncation-inequality} implies 
$
(q^m-q^{A_k-1})+ 
(q^{\widehat{A}^{(k)}_i} - q^{\widehat{A}^{(k)}_{i-1}+j}) 
\geq 
q^m.
$
Thus one has
$$
\begin{aligned}
t^{q^m-q^{A_k-1}} \qbin{m}{\widehat{\alpha}^{(k)}}{q,t}
&\equiv
t^{q^m-q^{A_k-1}}
\prod_{i=1}^{k} \prod_{j=0}^{\widehat{\alpha}^{(k)}_i-1}
\left(
1 - t^{q^{\widehat{A}^{(k)}_i}-q^{\widehat{A}^{(k)}_{i-1}+j}}
\right)^{-1}
\\
&=
\left(\prod_{i=1}^{k-1} \prod_{j=0}^{\alpha_i-1}
\left(
1 - t^{q^{A_i}-q^{A_{i-1}+j}}
\right)^{-1}
\right)
\cdot
t^{q^m-q^{A_k-1}}
\prod_{j=0}^{\alpha_k-2}
\left(
1 - t^{q^{A_k-1}-q^{A_{k-1}+j}}
\right)^{-1}.
\end{aligned}
$$
Using  \eqref{truncation-inequality}, the last, unparenthesized factor
is equivalent mod $(t^{q^m})$ to
$$
t^{q^m-q^{A_k-1}}
+ \sum_{j=0}^{\alpha_k-2}
t^{q^{m}-q^{A_{k-1}+j}}
=
\sum_{j=0}^{\alpha_k-1}
t^{q^{m}-q^{A_{k-1}+j}}.
$$
Consequently, one has
\begin{equation}
\label{final-form-for-k-summands}
t^{q^m-q^{A_k-1}} \qbin{m}{\widehat{\alpha}^{(k)}}{q,t}
\equiv
\left. 
\left(
\sum_{j=0}^{\alpha_k-1}
t^{q^{m}-q^{A_{k-1}+j}}
\right)
\middle/ 
\prod_{i=1}^{k-1} \prod_{j=0}^{\alpha_i-1}
(1 - t^{q^{A_i}-q^{A_{i-1}+j}})
\right. 
.
\end{equation}
Similarly one finds that
\begin{equation}
\label{first-summand-processing-start}
\qbin{m}{\widehat{\alpha}}{q,t}
=
\left.
\prod_{j=0}^{A_L-1} (1-t^{q^m-q^j})
\middle/ 
\prod_{i=1}^L \prod_{j=0}^{\alpha_i-1}
(1 - t^{q^{A_i}-q^{A_{i-1}+j}})
\right.
.
\end{equation}
The numerator on the right side of \eqref{first-summand-processing-start}
can be rewritten mod $(t^{q^m})$ using
 \eqref{truncation-inequality} as
$$
\prod_{j=0}^{A_L-1} (1-t^{q^m-q^j})
\equiv
1-\sum_{j=0}^{A_L-1} t^{q^m-q^j}
=
1-\sum_{k=1}^L \sum_{j=0}^{\alpha_k-1} t^{q^m-q^{A_{k-1}+j}}.
$$
Comparing this with
\eqref{parabolic-hilb-final-form}  shows that
\begin{equation}
\label{first-summand-final-form}
\begin{aligned}
\qbin{m}{\widehat{\alpha}}{q,t}
&=\Hilb(S^{P_{\alpha}},t)
-\sum_{k=1}^L
\left.
\left(
\sum_{j=0}^{\alpha_k-1} t^{q^m-q^{A_{k-1}+j}}
\right)
\middle/
\prod_{i=1}^L \prod_{j=0}^{\alpha_i-1}
(1 - t^{q^{A_i}-q^{A_{i-1}+j}})
\right.
 \\
&\equiv
\Hilb(S^{P_{\alpha}},t)
-\sum_{k=1}^L
\left.
\left(
\sum_{j=0}^{\alpha_k-1} t^{q^m-q^{A_{k-1}+j}}
\right)
\middle/
\prod_{i=1}^{k-1} \prod_{j=0}^{\alpha_i-1}
(1 - t^{q^{A_i}-q^{A_{i-1}+j}})
\right..
\end{aligned}
\end{equation}
The last equivalence mod $(t^{q^m})$ arises since if 
$i \geq k$ then $A_i \geq A_k$, so 
$
( q^m-q^{A_{k-1}+j} ) + (q^{A_i}-q^{A_{i-1}+j}) 
\geq q^m
$
by \eqref{truncation-inequality}.
Finally, combining
\eqref{only-summands-left}, \eqref{final-form-for-k-summands}, and
\eqref{first-summand-final-form} shows that
$
C_{\alpha,m}(t)
\equiv  
\Hilb(S^{P_{\alpha}},t),
$
as desired.

\section{Proofs in the bivariate case}
\label{n=2-section}

Our goal here is to prove Parabolic Conjectures~\ref{mod-powers-conjecture} 
and~\ref{cofixed-conjecture} for $n=2$.
Their equivalence for $n=2$ was shown in 
Corollary~\ref{n=2-equivalence-corollary}, so we only 
prove Parabolic Conjecture~\ref{cofixed-conjecture}.

The group $G=GL_2(\FF_q)$ has only
two parabolic subgroups $P_\alpha$, namely the whole group $G = P_{(2)}$ itself 
and the Borel subgroup $B = P_{(1,1)}$.  We establish Parabolic Conjecture~\ref{cofixed-conjecture} for these subgroups below in Theorems~\ref{bivariate-Stanley-decomposition} and~\ref{Stanley-decomposition-for-B-with-n=2}, respectively.

We consider the chain of subgroups 
\begin{equation}
\label{chain-of-subgroups}
\begin{array}{ccccccc}
1 &\subset& T &\subset& B& \subset& G\\
& & \Vert & &\Vert& &\Vert \\
& & \left\{
\left[\begin{matrix}
 a &  0\\
 0 & d
\end{matrix}\right]:
a,d \in \FF_q^\times
\right\}
& &
\left\{
\left[\begin{matrix}
 a &  b\\
 0 & d
\end{matrix}\right]:
a, d \in \FF_q^\times, b \in \FF_q
\right\}
& &
\left\{
\left[\begin{matrix}
 a &  b\\
 c & d
\end{matrix}\right]:
ad-bc \in \FF_q^\times
\right\}
\end{array}
\end{equation}
We first recall the known descriptions of the invariant subrings for each of these subgroups, and then prove some preliminary facts about their 
cofixed quotients.  Using this, we complete the analysis first for the 
quotient $S_B$, and finally for the quotient $S_G$.

\subsection{The invariant rings}
Acting on $S=\FF_q[x,y]$, the tower of subgroups \eqref{chain-of-subgroups}
induces a tower of invariant subalgebras
$S\supset S^T \supset S^B \supset S^G$, with the following explicit descriptions.
Abbreviate $X:=x^{q-1}, Y:=y^{q-1}$, and recall from the introduction
that for $n=2$ the two Dickson polynomials $D_{2,0}, D_{2,1}$ are defined by
\begin{equation}
\label{bivariate-Dickson-definition}
\prod_{(c_1,c_2) \in \FF_q^2 } (t+c_1 x+c_2 y) = t^{q^2} + D_{2,1} t^q + D_{2,0} t.
\end{equation}
\begin{proposition}
\label{bivariate-invariants-prop}
For $S=\FF_q[x,y]$ one has
\begin{itemize}
\item[(i)] $S^T=\FF_q[X,Y]$,
\item[(ii)] $S^B=\FF_q[X,D_{2,1}]$, and
\item[(iii)] $S^G=\FF_q[D_{2,0},D_{2,1}]$,
\end{itemize}
with explicit formulas
\begin{align*}
D_{2,1} &= Y^q + X Y^{q-1}  + \cdots +X^{q-1} Y + X^q, \\
D_{2,0} &= X Y^q + X^2 Y^{q-1} + \cdots  + X^q Y 
&= X D_{2,1} - X^{q+1}.
\end{align*}
\end{proposition}
\begin{proof}
Assertion (i) is straightforward.  Assertion (ii) follows from the
work of  Mui \cite{Mui} or Hewett \cite{Hewett}.   Assertion (iii) is Dickson's Theorem \cite{Dickson} for $n=2$.  The last two equalities follow from 
Dickson's expressions
\begin{equation}
\label{bivariate-Dicksons-explicitly} 
\begin{aligned}
D_{2,1}&=\left| \begin{matrix} x & y \\ x^{q^2} & y^{q^2} \end{matrix} \right|
\Bigg/
\left| \begin{matrix} x & y \\ x^{q} & y^{q} \end{matrix} \right| 
= \frac{xy^{q^2}-x^{q^2}y}{xy^q-x^qy} 
=Y^q + X Y^{q-1}  + \cdots +X^{q-1} Y + X^q\\
D_{2,0}&=
\left| \begin{matrix} x^q & y^q \\ x^{q^2} & y^{q^2} \end{matrix} \right| 
\Bigg/
\left| \begin{matrix} x & y \\ x^{q} & y^{q} \end{matrix} \right| 
= \frac{x^q y^{q^2}-x^{q^2}y^q}{xy^q-x^qy} 
=X Y^q + X^2 Y^{q-1}  + \cdots +X^q Y
\end{aligned}
\end{equation}
for the $D_{n,i}$ as quotients of determinants.
\end{proof}

\subsection{The cofixed spaces}
The tower of subgroups in \eqref{chain-of-subgroups} induces quotient maps
$
S \twoheadrightarrow 
S_T \twoheadrightarrow
S_B \twoheadrightarrow
S_G.
$
The quotient map $S \twoheadrightarrow S_T$ is easily understood.

\begin{proposition}
\label{T-fixed-quotient-prop}
A monomial $x^i y^j$ in $S$ survives in the 
$T$-cofixed space $S_T$ if and only if $q-1$ divides both $i$ and $j$, that is,
if and only if $x^i y^j=X^{i'} Y^{j'}$ for some $i'$, $j'$.
Furthermore these monomials $\{X^i Y^j\}_{i,j \geq 0}$ form an $\FF_q$-basis for
$S_T$.
\end{proposition}
\begin{proof}
Proposition~\ref{proposition: simple basis}(iv) implies that 
$S_T$ is the quotient of $S$ by the $\FF_q$-subspace spanned by
all elements $t(x^i y^j)-x^i y^j$.
A typical element $t$ in $T$ sends $x \mapsto c_1 x$ and $y \mapsto c_2 y$ for
some $c_1,c_2$ in $\FF_q^\times$.  Therefore 
$$
t(x^i y^j)-x^i y^j=(c_1^i c_2^j - 1) x^i y^j.
$$  
If both $i,j$ are divisible by $q-1$ then this will always be zero, and otherwise,
there exist choices of $c_1,c_2$ for which it is a nonzero multiple of $x^i y^j$.
\end{proof}

In understanding the quotients $S_P, S_G$, it helps
to define two $\FF_q$-linear functionals on $S$ that descend to 
one or both of $S_P$, $S_G$.  They are used in the proof of
Corollary~\ref{S_G-torsion-corollary} below to detect certain nonzero products.

\begin{definition}
\label{separating-functionals-definition}
Define two $\FF_q$-linear functionals $S \overset{\mu,\nu}{\longrightarrow} \FF_q$
by setting $\mu(x^i y^j)=\nu(x^i y^j)=0$ unless $q-1$ divides both $i,j$, and
setting
\begin{align*}
\mu(X^i Y^j) &= \begin{cases} 1 & \text{ if }i,j \geq 1,\\
                                                    0 &\text{ if }i=0 \text{ or } j=0\end{cases}\\
\intertext{and}
 \nu(X^i Y^j) &= \begin{cases} 1 &\text{ if }i=0,\\
                                                   0 & \text{ if }i \geq 1.\end{cases}
\end{align*}
\end{definition}
\noindent
In other words, $\mu$ applied to $f(x,y)$ sums the coefficients in $f$
on monomials of the form $X^i Y^j$ that are not pure powers $X^i$ or $Y^j$, while $\nu$ sums the coefficients on the pure $Y$-powers $Y^j$.  It should be clear from their definitions 
and Proposition~\ref{T-fixed-quotient-prop} that both 
$\mu, \nu$ descend to well-defined $\FF_q$-linear functionals on $S_T$.

\begin{proposition}
\label{functionals-that-descend}
One has the following:
\begin{itemize}
\item[(i)] The functional $S \overset{\nu}{\rightarrow} \FF$ 
descends to a well-defined functional on $S_B$.
\item[(ii)] The functional $S \overset{\mu}{\rightarrow} \FF$ 
descends to a well-defined functional on both
$S_B$ and $S_G$.
\end{itemize}
\end{proposition}
\begin{proof}
The Borel subgroup $B$ is generated by the torus $T$ together with a
transvection 
\begin{equation}
\label{favorite-transvection}
\begin{array}{rll}
x &\overset{u}{\longmapsto} x \\
y &\overset{u}{\longmapsto} x+y,
\end{array}
\end{equation}
while the full general linear group $G$ is generated by $B$ together with
a transposition $\sigma$ that swaps $x, y$.  Hence
by Proposition~\ref{proposition: simple basis}(iv),
it suffices to check that for every monomial $x^i y^j$, 
both $\mu, \nu$ vanish on 
\begin{equation}
\label{transvection-difference}
u(x^i y^j)-x^i y^j
=\sum_{k=0}^{j-1} \binom{j}{k} x^{i+j-k} y^{k}
\end{equation}
and that $\mu$ vanishes on 
\begin{equation}
\label{transposition-difference}
\sigma(x^i y^j)-x^i y^j=x^j y^i-x^i y^j.
\end{equation}

The fact that $\mu$ vanishes on \eqref{transposition-difference}
is clear from the symmetry between $X$ and $Y$ in
its definition.

To see that $\nu$ vanishes on 
\eqref{transvection-difference},
observe that $\nu$ vanishes on
every monomial $x^{i+j-k}y^k$ appearing in the sum,
as $k < j$ means it is never a pure power of $y$ (or $Y$).

To see that $\mu$ vanishes on \eqref{transvection-difference},
we do a calculation.
Applying $\mu$ to the right side gives
\begin{equation}
\label{transvection-functional-vanishing-check}
\sum_{k=0}^{j-1} \binom{j}{k} \mu(x^{i+j-k} y^{k}) 
 =\sum_{\substack{k=1,2,\ldots,j-1\\ q-1 \text{ divides }k}} \binom{j}{k}
\end{equation}
which equals the sum (in $\FF_q$) of the coefficients on the monomials
of the form $x^{\ell(q-1)}$ within the polynomial 
\[
f(x):=\sum_{k=1}^{j-1} \binom{j}{k} x^k
=(x+1)^j - (x^j + 1).
\]
One can then advantageously rewrite 
\eqref{transvection-functional-vanishing-check}
by taking advantage of a root of unity fact:
\[
 \sum_{\beta \in \FF_q^\times} \beta^k 
= \begin{cases} 
q-1 = -1 
  & \text{ if } k= \ell(q-1) \text{ for some }\ell \in \ZZ,\\
0 & \text{ otherwise}.
\end{cases}
\]
Noting also that $f(0)=0$, this lets one rewrite the right side of  \eqref{transvection-functional-vanishing-check} as
\begin{align*}
-\sum_{\beta \in \FF_q^\times} f(\beta)
=-\sum_{\beta \in \FF_q} f(\beta) 
&=-\sum_{\beta \in \FF_q} (\beta+1)^j 
     + \sum_{\beta \in \FF_q} \beta^j 
     + \sum_{\beta \in \FF_q} 1 \\
&=-\sum_{\beta \in \FF_q} \beta^j + \sum_{\beta \in \FF_q} \beta^j + q  = 0. \qedhere
\end{align*}
\end{proof}

The following technical lemma on vanishing and equalities lies 
at the heart of our analysis of $S_B$, $S_G$.
\begin{lemma}
\label{equal-monomomials-in-quotients}
Beyond the vanishing in $S_T$ of monomials except for $\{X^i Y^j\}_{i,j \geq 0}$,
in the further quotient $S_B$ one also has
\begin{enumerate}
\item[(i)] $X^i=0$ for all $i \geq 1$,
\item[(ii)] $X^i Y^j = X^{i'} Y^{j'}$ for all $i,i' \geq 1$ and $1 \leq j,j' \leq q$ if $i+j=i'+j'$.
\end{enumerate}
In the even further quotient $S_G$, one additionally has
\begin{enumerate}
\item[(iii)] $Y^j=0$ for all $j \geq 1$, and
\item[(iv)] $X^i Y^j= X^{i'} Y^{j'}$ for all $i,i',j,j' \geq 1$ with $i+j=i'+j' \leq 2q$.
\end{enumerate}
\end{lemma}

\begin{proof}
For (i), since $B$ contains the transvection $u$ from
\eqref{favorite-transvection}, one has in $S_B$ for any $k > 0$ that
$$
0 \equiv u(x^{k-1} y)-x^{k-1}y=
x^{k-1}(x+y) - x^{k-1}y =
x^k.
$$
Hence $X^i=x^{i(q-1)}$ vanishes in $S_B$ for all $i > 0$.

For (ii), we claim that it suffices to show that 
whenever $i,j \geq 1$ and $2 \leq j \leq q$,
one can express $X^i Y^j$ as a sum of $X^{i'} Y^{j'}$ having $i+j=i'+j'$ and
$j' < j$:  then all such monomials $X^i Y^j$ will be scalar multiples
of each other, but they all take the same value $1$ when one applies the
functional $\mu$ from Definition~\ref{separating-functionals-definition}
and Proposition~\ref{functionals-that-descend}, so they must all be equal.  

To this end, let $d:=(i+j)(q-1)=\deg(X^i Y^j)$.
Using the transvection $u$ from \eqref{favorite-transvection},
and taking advantage of the
vanishing of $x^i y^j$ in $S_B$ unless $q-1$ divides $i,j$, one has
\begin{align*}
0 &\equiv u(x^{d-(jq-1)} y^{jq-1})
    -x^{d-(jq-1)} y^{jq-1} \\
&=x^{d-(jq-1)} (x+y)^{jq-1}
    -x^{d-(jq-1)} y^{jq-1}\\
&=\left( \sum_{k=0}^{jq-1} \binom{jq-1}{k} x^{d-k} y^k \right) 
    -x^{d-(jq-1)} y^{jq-1}\\
&\equiv \binom{jq-1}{j(q-1)} x^{i(q-1)} y^{j(q-1)} 
        + \sum_{m=0}^{j-1} \binom{jq-1}{m(q-1)} x^{(i+j-m)(q-1)} y^{m(q-1)} \\
&= \binom{jq-1}{j(q-1)} X^i Y^j
        + \sum_{m=0}^{j-1} \binom{jq-1}{m(q-1)} X^{i+j-m} Y^m.
\end{align*}

Thus it remains only to show that 
$\binom{jq-1}{j(q-1)} \neq 0$ in $\FF_q$ when $1 \leq j \leq q$.  Letting
$q=p^s$ for some prime $p$ and exponent $s \geq 1$, 
we have
\begin{equation}\label{lucas equation}
\binom{jq - 1}{j(q - 1)} = \frac{(jq - 1)(jq - 2) \cdots (jq - j + 1)}{1\cdot 2 \cdots (j - 1)}.
\end{equation}
For any integers $a, b$ such that $1 \leq a \leq p^s - 1$ and $b \geq 1$, the largest power of $p$ dividing $b\cdot p^s - a$ is equal to the largest power of $p$ dividing $a$.  Since $j \leq q$, it follows that the largest power of $p$ dividing the numerator of the right side of \eqref{lucas equation} is equal to the largest power of $p$ dividing the denominator, so $\binom{jq - 1}{j(q - 1)} \neq 0$ in $\FF_q$.

For (iii), note that since (i) implies $X^i$ vanishes in $S_B$, 
the same vanishing holds in the further quotient $S_G$.
But then $Y^i$ also vanishes in $S_G$ by applying the transposition $\sigma$
in $G$ swapping $x, y$.

For (iv), note that (ii) shows that, fixing $d:=i+j$,
all monomials $X^i Y^j$ with $i,j \geq 1$ and $j \leq q$ are equal in $S_B$,
and hence also equal in the further quotient $S_G$.  Applying
the transposition $\sigma$ as before, one concludes
that these monomials are also all equal to the monomials
$X^i Y^j$ with $i,j \geq 1$ and $i \leq q$.  But when $d=i+j \leq 2q$ 
these two sets of monomials exhaust all of the possibilities
for $X^i Y^j$ with $i,j \geq 1$.
\end{proof}

The following corollary will turn out to be a crucial 
part of the structure of $S_G$ as an $S^G$-module in the bivariate case, 
used in the proof of 
Theorem~\ref{bivariate-Stanley-decomposition} below.  

\begin{corollary}
\label{S_G-torsion-corollary}
In the $G$-fixed quotient space $S_G$, the images of the monomials
\begin{equation}
\label{nonfree-G-cofixed-basis-elements}
\{1, \,\, XY, \,\, X^2Y, \,\, \ldots, \,\, X^{q-2}Y\}
\end{equation}
are all annihilated by $D_{2,0}$, but
none of them is annihilated by any power of $D_{2,1}$.
\end{corollary}
\begin{proof}
Proposition~\ref{bivariate-invariants-prop} shows that
$D_{2,0}$ is a sum of $q$ monomials of the form
$X^i Y^j$ with $i,j \geq 1$.  The same is true for the product
$D_{2,0} \cdot M$ where $M$ is any of the monomials in
\eqref{nonfree-G-cofixed-basis-elements}.  Since these monomials
$M$ have degree at most $(q-1)^2$, the product 
$D_{2,0} \cdot M$ has degree at most $q^2-1+(q-1)^2=2q(q-1)$,
and hence all $q$ of the monomials in the product are equal to
the same monomial $M'$ by Lemma~\ref{equal-monomomials-in-quotients}(iv).
Therefore $D_{2,0} M \equiv qM'=0$ in $S_G$, as desired.

Proposition~\ref{bivariate-invariants-prop} 
shows that $D_{2,1}=Y^q+XY^{q-1}+\cdots+X^{q-1}Y+X^q$,
a sum of $q+1$ monomials. Hence for $j \geq 0$, 
the power $D_{2,1}^j$ is a sum of $(q+1)^j$ monomials, of the form
$$
D_{2,1}^j = Y^{qj} + \left( \sum_{i,j \geq 1} c_{i,j} X^i Y^j \right) + X^{qj} 
$$
with $\sum_{i,j \geq 1} c_{i,j} = (q+1)^j-2$.
Thus the $\FF_q$-linear functional $\mu$ 
from Definition~\ref{separating-functionals-definition} 
and Proposition~\ref{functionals-that-descend}
has
$$
\mu(D_{2,1}^j \cdot 1) = \mu(D_{2,1}^j)=(q+1)^j-2 = 1^j-2=-1 \neq 0,
$$
while for any of the rest of the monomials $M=X^iY$ with $i \geq 1$ 
in \eqref{nonfree-G-cofixed-basis-elements}, it has
$$
\mu(D_{2,1} \cdot M) = (q+1)^j = 1^j =1 \neq 0.
$$
Thus no power $D_{2,1}^j$ annihilates any of the monomials in 
\eqref{nonfree-G-cofixed-basis-elements} within $S_G$.
\end{proof}

\subsection{Analyzing the fixed quotient $S_B$ for the
Borel subgroup $B=P_{(1,1)}$}

One can regard the polynomial algebra $S$ with its $B$-action as
a module for the group algebra $S^B[B]$ having coefficients in the
$B$-invariant subalgebra $S^B=\FF_q[D_{2,1},X]$.  
We begin by describing 
the $S^B[B]$-module structure on $S$, 
and thereby deduce the $S^B$-module structure on the $B$-cofixed space $S_B$.
For this purpose, we borrow an idea from 
Karagueuzian and Symonds \cite[\S 2.1]{KaragueuzianSymonds1}.

\begin{definition}
\label{hat-S-definition}
Let $\widehat{S}$ be the $\FF_q$-subspace of $S$ spanned by the monomials 
$\{x^i y^j: 
0 \leq j \leq q^2-q\}$.
\end{definition}

\noindent
It is easily seen that $\widehat{S}$ is stable under the action of
$B$, and also under multiplication by $x$ and so by its $B$-invariant power
$X=x^{q-1}$, so that $\widehat{S}$ becomes an $\FF_q[X][B]$-module.
Thus the tensor product 
$$
\FF_q[D_{2,1}] \otimes_{\FF_q} \widehat{S}
$$
is naturally a module for the ring
$$
\FF_q[D_{2,1}] \otimes_{\FF_q} \FF_q[X][B] \quad \cong \quad \FF_q[D_{2,1},X][B]=S^B[B]
$$
via the tensor product action 
$$
(a \otimes c)(b \otimes d)= ab \otimes cd
$$ 
for any elements 
\[
a,b \in \FF_q[D_{2,1}], \qquad
c \in  \FF_q[X][B], \quad \text{ and } \quad
d\in \widehat{S}.
\]
\begin{proposition}[{cf. \cite[Lemma 2.5]{KaragueuzianSymonds1}}]
\label{Karagueuzian-Symonds-prop}
The multiplication map
$$
\begin{array}{rcccl}
\FF_q[D_{2,1}] &\otimes_{\FF_q}& \widehat{S} & \longrightarrow & S  \\
f_1 &\otimes& f_2 & \longmapsto & f_1 f_2
\end{array}
$$
induces an $S^B[B]$-module isomorphism.
Hence as a module over  $S^B=\FF_q[D_{2,1},X]$, 
one has an isomorphism
$$
\FF_q[D_{2,1}] \otimes_{\FF_q} \widehat{S}_B \cong S_B.
$$
\end{proposition}
\begin{proof}
The multiplication map is easily seen to be a morphism of
$S^B[B]$-modules, so it remains only to check that it is 
an $\FF_q$-vector space isomorphism.  This follows by iterating 
a direct sum decomposition 
\begin{equation}
\label{Karagueuzian-Symonds-style-direct-sum}
D_{2,1} S_d \quad \oplus \quad \widehat{S}_{d+q^2-q} = S_{d+q^2-q} 
\end{equation}
justified for $d \geq 0$ as follows.  
The leftmost summand $D_{2,1} S_d$ in
\eqref{Karagueuzian-Symonds-style-direct-sum}
has as $\FF_q$-basis the set
$\{D_{2,1} x^i y^j\}_{i+j=d}$. Since
\eqref{bivariate-Dicksons-explicitly} shows that
$
D_{2,1}=y^{q^2-q} + x^{q-1} y^{q^2-2q+1} + \cdots + x^{q^2-q},
$ 
the leading monomials in $y$-degree for
$D_{2,1} S_d$ are 
$$
\{ x^i y^{j'}: i+j'=d+q^2-q \text{ and } j' \geq q^2-q\}.
$$ 
Meanwhile the summand
$\widehat{S}_{d+q^2-q}$ has as $\FF_q$-basis 
the complementary set of monomials 
$$
\{ x^i y^j: i+j=d+q^2-q \text{ and } j < q^2-q\}
$$
within the set of all monomials 
$
\{ x^i y^j: i+j=d+q^2-q \} 
$
that form an  $\FF_q$-basis for $S_{d+q^2-q}$.
\end{proof}

In analyzing $S_B$, it therefore suffices to analyze $\widehat{S}_B$.

\begin{proposition}
\label{n=2-parabolic-cofixeds}
Within the quotient space $\widehat{S}_B$, one has the following.
\begin{enumerate}
\item[(i)]
$X^i Y^j \equiv X^{i+j-1} Y$ 
for all $i \geq 1$ and $1 \leq j \leq q-1$.
\item[(ii)] There is an $\FF_q$-basis
\begin{equation}
\label{hat-cofixed-basis}
\{ Y, XY, X^2Y, X^3Y,\ldots \} \cup \{1,Y^2,Y^3,\ldots,Y^{q-1}\}.
\end{equation}
\item[(iii)]
There is an $\FF_q[X]$-module direct sum decomposition
$\widehat{S}_P=M_1 \oplus M_2$, where 
\begin{itemize}
\item 
$M_1=\FF_q[X] \cdot Y$ is a free $\FF_q[X]$-module on the basis $\{Y\}$,
and 
\item $M_2$ is the $\FF_q[X]$-submodule spanned by
\begin{equation}
\label{nonfree-parabolic-cofixed-basis-elements}
\{1, \,\, Y^2-XY, \,\, Y^3-X^2Y, \,\, \ldots, \,\, Y^{q-1}-X^{q-2}Y\},
\end{equation}
having $\FF_q[X]$-module structure isomorphic to a
direct sum of copies of the quotient module
$\FF_q[X]/(X)$ with the elements 
of \eqref{nonfree-parabolic-cofixed-basis-elements} as basis.
\end{itemize}
\end{enumerate}
\end{proposition}

\begin{proof}
\vskip.1in
\noindent
{\sf Assertion (i).}
This follows from Lemma~\ref{equal-monomomials-in-quotients}(ii).

\vskip.1in
\noindent
{\sf Assertion (ii).}
We first argue that the  monomials in
\eqref{hat-cofixed-basis}
span $\widehat{S}_B$.  By Definition~\ref{hat-S-definition}, one has that $\widehat{S}$ is 
$\FF_q$-spanned by 
$
\{x^i y^j: i\geq 0 \text{ and  }0 \leq j < q^2-q\}.
$
Since monomials other than those 
of the form $X^i Y^j$ vanish in $S_T$ and
thus in its further quotient $S_B$, one concludes that $\widehat{S}_B$ is
$\FF_q$-spanned by 
$$
\{X^i Y^j: i\geq 0 \text{ and }0 \leq j < q\}.
$$
Lemma~\ref{equal-monomomials-in-quotients}(i) says 
$X^i$ vanishes in $S_B$ for $i\geq 1$, 
so one may discard these monomials and still have a spanning set.  
Also, assertion (i) of the present proposition shows 
that one may further discard monomials of the form 
$X^i Y^j$ with $i \geq 1$ and $j > 1$. 
Thus $\widehat{S}_B$ is $\FF_q$-spanned by 
$$
\{X^i Y\}_{i\geq 0} \cup \{Y^j\}_{0 \leq j \leq q-1},
$$
which is the same set as in \eqref{hat-cofixed-basis}.

To see that these monomials
are $\FF_q$-linearly independent in $\widehat{S}_B$ or $S_B$, this
table shows that they are
separated in each degree by the $\FF_q$-linear
functionals $\mu$ and $\nu$ on $S_B$ from 
Definition~\ref{separating-functionals-definition}
and Proposition~\ref{functionals-that-descend}:

\begin{center}
\begin{tabular}{|c||c|c|c c|c c|c|c c|c|c|c|c|} \hline
degree   & $0$   & $1$  & $2$ &   & $3$ &  & $\cdots$ 
         & $q-1$     &   & $q$       & $q+1$ & $q+2$ & $\cdots$ \\ \hline
monomial & $1$     & $Y$    & $XY$  &$Y^2$  & $X^2Y$&$Y^3$  & $\cdots$
         & $X^{q-2}Y$ & $Y^{q-1}$& $X^{q-1}Y$ & $X^{q-2}Y$ & $X^{q-3}Y$ & $\cdots$ \\ \hline
$\mu$ value& $0$     & $0$    & $1$  &$0$  & $1$&$0$  & $\cdots$
         & $1$ & $0$& $1$ & $1$ & $1$ & $\cdots$ \\\hline
$\nu$ value& $1$     & $1$    & $0$  &$1$  & $0$&$1$  & $\cdots$
         & $0$ & $1$& $0$ & $0$ & $0$ & $\cdots$ \\ \hline
\end{tabular}
\end{center}

\noindent
{\sf Assertion (iii).}
First note that
since $\{Y,XY,X^2Y,X^3Y,\ldots\}$ is a subset of an $\FF_q$-basis for
$\widehat{S}_B$, the submodule $M_1=\FF_q[X] \cdot Y$ indeed forms
a free $\FF_q[X]$-module on the basis $\{Y\}$ inside of
$\widehat{S}_B$.  Since 
$$
\{1\} \cup \{Y^j\}_{2 \leq j \leq q-1}
$$ 
extends 
$\{Y,XY,X^2Y,X^3Y,\ldots\}$ to an $\FF_q$-basis for $\widehat{S}_B$,
so does the set \eqref{nonfree-parabolic-cofixed-basis-elements}
$$
\{1\} \cup \{Y^j-X^{j-1}Y\}_{2 \leq j \leq q-1}.
$$
In particular, none of these elements vanish in $\widehat{S}_B$, 
and $\widehat{S}_B = M_1 + M_2$ where
$M_2$ is the $\FF_q[X]$-span of 
\eqref{nonfree-parabolic-cofixed-basis-elements}.
On the other hand, each element of 
\eqref{nonfree-parabolic-cofixed-basis-elements}
is annihilated on multiplication by $X$:  this holds for the monomial 
$1$ since $X$ vanishes in $S_B$ by 
Lemma~\ref{equal-monomomials-in-quotients}(i), and it holds for
$Y^j-X^{j-1}Y$ with $2 \leq j \leq q-1$
since $XY^j \equiv X^jY$ in $S_B$ by 
Lemma~\ref{equal-monomomials-in-quotients}(ii).
Thus $M_2$ has \eqref{nonfree-parabolic-cofixed-basis-elements}
as an $\FF_q$-basis, and its $\FF_q[X]$-module
structure is that of a free $\FF_q[X]/(X)$-module on this
same basis.  This also shows that one has a {\it direct} sum 
$\widehat{S}_B = M_1 \oplus M_2$.
\end{proof}

The following is immediate from Proposition
\ref{Karagueuzian-Symonds-prop} and \ref{n=2-parabolic-cofixeds}.

\begin{theorem}
\label{Stanley-decomposition-for-B-with-n=2}
One has a direct sum decomposition 
$S_B=M_1' \oplus M_2'$ as modules for $S^B=\FF_q[D_{2,1},X]$, where
\begin{itemize}
\item 
$M_1'$ is a free $\FF_q[D_{2,1},X]$-module on $\{Y\}$, and 
\item
$M_2'$ is a direct sum of copies of the quotient $S^B$-module
$\FF_q[D_{2,1},X]/(X)$ with basis listed in 
\eqref{nonfree-parabolic-cofixed-basis-elements}.
\end{itemize}
In particular, one has
$$
\Hilb(S_B,t) = 
 \frac{t^{q-1}}{(1-t^{q-1})(1-t^{q^2-q})} 
  + \frac{1 + t^{2(q-1)} + t^{3(q-1)} + \cdots + t^{(q-1)^2} }{1-t^{q^2-q}} 
$$
which equals the prediction from
Parabolic Conjecture~\ref{cofixed-conjecture} for $\alpha=(1,1)$, namely
\[
\Hilb(S_B,t)=1  + \frac{t^{q-1}}{1-t^{q-1}} 
   + \frac{t^{2(q-1)}}{1-t^{q-1}} 
   + \frac{t^{q^2+q-2}}{(1-t^{q-1})(1-t^{q^2-q})},
\]
with the four summands corresponding to $\beta=(0,0)$, $(0,1)$, $(1,0)$, $(1,1)$, respectively.
\end{theorem}

\subsection{Analyzing the fixed quotient $S_G$ 
for the full group $G=GL_2(\FF_q)=P_{(2)}$}

One can again regard the polynomial algebra $S$ with its $G$-action as
a module for the group algebra $S^G[G]$ with coefficients in the
$G$-invariant subalgebra $S^G=\FF_q[D_{2,0},D_{2,1}]$.  Our strategy here
in understanding $S_G$ as an $S^G$-module differs from the previous section, 
as we do not have a $G$-stable subspace in $S$ acted on
freely by $D_{2,1}$ to
play the role of the $B$-stable subspace $\widehat{S} \subset S$.
Instead we will work with quotients by $D_{2,1}$.  

\begin{proposition}
\label{D21-quotient-in-either-order}
One has an $S^G$-module isomorphism 
$
\left( S/(D_{2,1}) \right)_G 
\cong 
S_G/ D_{2,1} S_G.
$
\end{proposition}
\begin{proof}
Both are isomorphic to
$
S/( D_{2,1}S + \spanof_{\FF_q}\{g(f) - f\}_{g \in G, f \in S} ).
$
\end{proof}

We wish to first analyze $\left( S/(D_{2,1}) \right)_G$ as
an $S^G$-module.  For this it helps that we already understand 
$\left( S/(D_{2,1}) \right)_B$ as an $S^B$-module, due to the following result.

\begin{proposition}
\label{hat-S-as-section-prop}
The composite map
$
\widehat{S} \hookrightarrow S \twoheadrightarrow S/(D_{2,1})
$
is an isomorphism of $\FF_q[X][B]$-modules, which then induces
an isomorphism of $\FF_q[X]$-modules
$
\widehat{S}_B \cong  (S/(D_{2,1}))_B.
$
\end{proposition}
\begin{proof}
The first assertion comes from 
Proposition~\ref{Karagueuzian-Symonds-prop},
and the second assertion follows from the first. 
\end{proof}

\begin{proposition}
\label{D21-quotient-analysis-prop}
The set
\begin{equation}
\label{D21-quotient-generators}
\{1, \,\, XY, \,\, X^2Y, \,\, \ldots, \,\, X^{q-2}Y\} \cup \{ X^q Y\}
\end{equation}
generates $S_G/D_{2,1}S_G$ as a module over $\FF_q[D_{2,0}]$,
and hence generates $S_G$ as module over $\FF_q[D_{2,0},D_{2,1}]=S^G$.
\end{proposition}
\begin{proof}
The second assertion follows from the first via this
well-known general lemma.

\begin{lemma}
Let $R$ be an $\NN$-graded ring. 
Let $I \subset R_+:=\bigoplus_{d >0} R_d$ be a homogeneous ideal of positive
degree elements.  
Let $M$ be a $\ZZ$-graded $R$-module with nonzero degrees 
bounded below.

Then a subset generates $M$ as an $R$-module
if and only if its images generate $M/IM$  as $R/I$-module.
\end{lemma}
\begin{proof}[Proof of lemma]
The ``only if'' direction is clear.  For the ``if'' direction, one assumes that $\{m_i\}$ in $M$
have $\{m_i+IM\}$ generating $M/IM$ as $R/I$-module, and shows that every
homogeneous element $m$ in $M$ lies in $\sum_i Rm_i$
via a straightforward induction on the degree of $m$.
\end{proof}

Returning to the proof of the first assertion in the proposition, we use
Proposition~\ref{D21-quotient-in-either-order} to
work with $(S/(D_{2,1}))_G$ rather than $S_G/D_{2,1}S_G$.
As noted in Proposition~\ref{bivariate-invariants-prop},
$D_{2,0}=XD_{2,1}-X^{q+1}$, and hence
$$
D_{2,0} \equiv -X^{q+1} \bmod{(D_{2,1})}.
$$
Thus via the quotient  map 
$(S/(D_{2,1}))_B \twoheadrightarrow (S/(D_{2,1}))_G$,
one obtains an $\FF_q[D_{2,0}]$-spanning
set for $(S/(D_{2,1}))_G$ from any $\FF_q[X^{q+1}]$-spanning set 
of $(S/(D_{2,1}))_B$, or equivalently via 
Proposition~\ref{hat-S-as-section-prop}, from any
$\FF_q[X^{q+1}]$-spanning set of $\widehat{S}_B$.
Since $\widehat{S}_B$ has as $\FF_q$-basis the monomials 
$\{X^i Y\}_{i \geq 0} \cup \{1,Y^2,Y^3,\ldots,Y^{q-1}\}$
from \eqref{hat-cofixed-basis}, it has as an
$\FF_q[X^{q+1}]$-spanning set
\[
\{X^i Y\}_{0 \leq i \leq q}  \cup \{1,Y^2,Y^3,\ldots,Y^{q-1}\}.
\]
Thus, this set is an $\FF_q[D_{2,0}]$-spanning
set for $(S/(D_{2,1}))_G$.  
However, Lemma~\ref{equal-monomomials-in-quotients}(iii) says that
the pure powers $\{Y^j\}_{j \geq 1}$ all vanish in $S_G$, so one obtains
this smaller  $\FF_q[D_{2,0}]$-spanning set for $(S/(D_{2,1}))_G$:
$$
\{1, \,\, X Y, \,\,  X^2Y, , \,\, \ldots,  
\,\, X^{q-2}Y, \,\,  X^{q-1}Y, \,\,  X^qY \}.
$$
We claim that the second-to-last element 
$X^{q-1}Y$ on this list is also redundant,
as it vanishes in $(S/(D_{2,1}))_G$.  To see this claim, note that
in $(S/(D_{2,1}))_G$ one has
$$
0 \equiv D_{2,1} = Y^q + (XY^{q-1} + X^2Y^{q-2} 
                    + \cdots + X^{q-2}Y^2 + X^{q-1}Y) + X^q.
$$
Here the two pure powers $X^q, Y^q$ vanish in $S_G$ and also
in $(S/(D_{2,1}))_G$ due to 
Lemma~\ref{equal-monomomials-in-quotients}(i),(iii).
Similarly, the $q-1$ monomials inside the parenthesis 
$X^i Y^{q-i}$ for $i=1,2,\ldots,q-1$ are all equal to $X^{q-1}Y$ due to 
Lemma~\ref{equal-monomomials-in-quotients}(i).
This implies $0\equiv(q-1)X^{q-1}Y=-X^{q-1}Y$ as claimed.
\end{proof}

\begin{theorem}
\label{bivariate-Stanley-decomposition}
One has an $S^G$-module direct sum decomposition 
$S_G = N_1 \oplus N_2$, in which
\begin{itemize}
\item
$N_1=S^G \cdot X^qY$ is a free $S^G$-module on the basis $\{X^q Y\}$,
and 
\item $N_2$ is the $S^G$-submodule spanned by the elements
of \eqref{nonfree-G-cofixed-basis-elements},
whose $S^G$-module structure is a direct sum of $q-1$ copies of $S^G/(D_{2,0})$
with the elements of \eqref{nonfree-G-cofixed-basis-elements} as basis.
\end{itemize}
In particular, in the bivariate case $n=2$, 
Question~\ref{Stanley-decomposition-conjecture} has an affirmative answer, 
and one has
\begin{align*}
\Hilb(S_G,t)
&=
\frac{t^{q^2-1}}{(1 - t^{q^2 - 1})(1 - t^{q^2 - q})} 
+
\frac{1 + t^{2(q-1)} + t^{3(q-1)} + \cdots + t^{(q-1)^2}}
{1 - t^{q^2 - q}} \\
&=1 + \frac{t^{2(q - 1)}}{1 - t^{q - 1}} +
\frac{t^{2(q^2 - 1)}}{(1 - t^{q^2 - 1})(1 - t^{q^2 - q})},
\end{align*}
so that Conjecture~\ref{cofixed-conjecture} holds.
\end{theorem}

\begin{proof}
Define $N_1$, $N_2$ 
to be the $S^G$ submodules of $S_G$ spanned by $\{X^qY\}$
and of the elements of \eqref{nonfree-G-cofixed-basis-elements}, respectively.
Then Proposition~\ref{D21-quotient-analysis-prop} 
implies $S^G = N_1 + N_2$.  Note that 
Corollary~\ref{S_G-torsion-corollary} already shows
that the submodule $N_2$ has the claimed structure.
In particular, $D_{2,0} \cdot N_2=0$, that is, $N_2 \subset \Ann_{S_G}{D_{2,0}}$.

We claim that this forces $N_2=S^G \cdot X^qY \cong S^G$, that is,
no element $f$ in $S^G$ can annihilate $X^q Y$.  
Otherwise, there would be an element $D_{2,0} f$ in $S^G$
annihilating both $N_1, N_2$, and hence annihilating all of
$S_G$, contradicting the assertion from Proposition~\ref{rank-one-prop} that
$S_G$ is a rank one $S^G$-module.

Once one knows $N_2=S^G \cdot X^qY \cong S^G$, one can also conclude
that the sum $S_G=N_1 + N_2$ is direct, since
\[
N_1 \cap N_2 \subset \Ann_{S_G}(D_{2,0}) \cap N_2 = 0. \qedhere
\]
\end{proof}

\begin{remark}
\label{Karagueuzian-Symonds-remark}
Our proof for Parabolic 
Conjectures~\ref{mod-powers-conjecture} and \ref{cofixed-conjecture}
with $n=2$ is hands-on and technical.
One might hope to use more of the results of Karagueuzian and Symonds
\cite{KaragueuzianSymonds1, KaragueuzianSymonds2,  KaragueuzianSymonds3}.
They give a good deal of information about the action of
$G=GL_n(\FF_q)$ on $S=\FF_q[x_1,\ldots,x_n]$, by analyzing
in some detail the structure of $S$ as an $\FF_q U$-module,
where $U$ is the $p$-Sylow subgroup of $G$ consisting of all unipotent
upper-triangular matrices. 
We have not seen how to apply this toward resolving
our conjectures in general.
\end{remark}


\begin{thebibliography}{99}


\bibitem{AkinBuchsbaumWeyman}
K. Akin, D.A. Buchsbaum, and J. Weyman, 
Schur functors and Schur complexes.
{\it Adv. in Math.} {\bf 44} (1982),  207--278.
 


\bibitem{ArmstrongRhoadesR}
D. Armstrong, V. Reiner and B. Rhoades,
Parking spaces, {\tt arXiv:1204.1760}.

\bibitem{AtiyahMacdonald}
M.F. Atiyah and I.G. Macdonald,
Introduction to commutative algebra. 
Addison-Wesley Publishing Co., 1969.

\bibitem{BalagovicChen}
M. Balagovi\'c and H. Chen, 
Representations of rational Cherednik algebras 
in positive characteristic.
{\it J. Pure Appl. Algebra} {\bf 217} (2013), 716--740. 

\bibitem{Benson}
D.J. Benson,
Polynomial invariants of finite groups. 
{\it London Mathematical Society Lecture Note Series} {\bf 190}. 
Cambridge University Press, Cambridge, 1993.


\bibitem{BourguibaZarati}
D. Bourguiba and S. Zarati, 
Depth and the Steenrod algebra.
With an appendix by J. Lannes.
{\it Invent. Math.} {\bf 128} (1997), 589--602. 

\bibitem{CarlisleKuhn}
D.P.~Carlisle and N.J.~Kuhn, 
Subalgebras of the Steenrod algebra and the action of matrices
on truncated polynomial algebras, 
{\it J. Algebra} {\bf 121} (1989) 370--387.


\bibitem{Dickson}
L.E.~Dickson, 
A fundamental system of invariants of the general modular linear group 
with a solution of the form problem.
{\it Trans. Amer. Math. Soc.} {\bf 12} (1911), 75--98.


\bibitem{DotyWalker1}
S. Doty and G. Walker,
The composition factors of ${\bf F}_p[x_1,x_2,x_3]$ as a ${\rm GL}(3,p)$-module. J. Algebra 147 (1992), no. 2, 411--441. 

\bibitem{DotyWalker2}
S. Doty and G. Walker,
Modular symmetric functions and irreducible modular representations of general linear groups. J. Pure Appl. Algebra 82 (1992), no. 1, 1--26. 

\bibitem{DotyWalker3}
S. Doty and G. Walker,
Truncated symmetric powers and modular representations of ${\rm GL}_n$. 
{\it Math. Proc. Cambridge Philos. Soc.} {\bf  119} (1996), 231--242. 

\bibitem{DummitFoote}
D.S. Dummit and R.M. Foote,
Abstract algebra. Third edition. 
John Wiley \& Sons, Inc., Hoboken, NJ, 2004. 


\bibitem{GoldmanRota}
J. Goldman and G.-C. Rota, 
The number of subspaces of a vector space. 
Recent Progress in Combinatorics (Proc. Third Waterloo Conf. on Combinatorics, 1968),
pp. 75--83. Academic Press, New York, 1969.


\bibitem{HartmannShepler}
J. Hartmann and A.V. Shepler,
Reflection groups and differential forms.
{\it Math. Res. Lett.} {\bf 14} (2007), 955--971.



\bibitem{Hewett}
T.J. Hewett, 
Modular invariant theory of parabolic subgroups of 
${\rm GL}\sb n( F\sb q)$ and the associated Steenrod modules. 
{\it Duke Math. J.} {\bf 82} (1996), 91--102;
with a correction in {\it Duke Math. J.} {\bf 97} (1999), 217.



\bibitem{KaragueuzianSymonds1}
D. Karagueuzian and P. Symonds, 
The module structure of a group action on a polynomial ring. 
{\it J. Algebra} {\bf 218} (1999), no. 2, 672--692. 
 
\bibitem{KaragueuzianSymonds2}
D. Karagueuzian and P. Symonds,
The module structure of a group action on a polynomial ring: 
examples, generalizations, and applications. 
Invariant theory in all characteristics, 139--158, 
{\it CRM Proc. Lecture Notes} {\bf 35}, Amer. Math. Soc., Providence, RI, 2004. 

\bibitem{KaragueuzianSymonds3}
D. Karagueuzian and P. Symonds
The module structure of a group action on a polynomial ring: 
a finiteness theorem. 
{\it J. Amer. Math. Soc.} {\bf 20} (2007), 931--967.

\bibitem{Kuhn}
N.J. Kuhn, 
The Morava K-theories of some classifying spaces.
{\it Trans. Amer. Math. Soc.} {\bf 304} (1987), 193--205. 

\bibitem{KuhnMitchell}
N. Kuhn and S. Mitchell,
The multiplicity of the Steinberg representation of ${\rm GL}\sb n F\sb q$ in the symmetric algebra.
{\it Proc. Amer. Math. Soc.} {\bf 96} (1986), no. 1, 1--6.


\bibitem{LandweberStong}
P. S. Landweber and R. E. Stong, 
The depth of rings of invariants over finite fields,
{\it Springer Lect. Notes in Math.} {\bf 1240} (1987), 259--274.

\bibitem{Lang}
 S. Lang, 
 Algebra. Revised third edition.
 {\it Graduate Texts in Mathematics} {\bf  211}. Springer-Verlag, New York, 2002.


\bibitem{Mui}
H. Mui,
Modular invariant theory and cohomology algebras of symmetric groups.  
{\it J. Fac. Sci. Univ. Tokyo Sect. IA Math.}  {\bf 22}  (1975), no. 3, 319--369

\bibitem{StantonR}
V. Reiner and D. Stanton,
$(q,t)$-analogues and $GL_n(\FF_q)$. 
{\it J. Algebraic Combin.}{\bf 31} (2010), 411--454. 


\bibitem{StantonWhiteR}
V. Reiner, D. Stanton, and D. White,
The cyclic sieving phenomenon. 
{\it J. Combin. Theory Ser. A} {\bf 108} (2004), 17--50. 


\bibitem{Rhoades}
B. Rhoades,
Parking structures:  Fuss analogs,
{\tt arXiv:1205.4293}.

\bibitem{Segal}
J. Segal,
Notes on invariant rings of divided powers. 
Invariant theory in all characteristics, 229--239,
{\it CRM Proc. Lecture Notes} {\bf 35}, Amer. Math. Soc., Providence, RI, 2004. 

\bibitem{Serre}
J.-P. Serre, 
Linear representations of finite groups.
{\it Graduate Texts in Mathematics} {\bf Vol. 42}. 
Springer-Verlag, New York-Heidelberg, 1977.

\bibitem{Smith}
L. Smith,
Polynomial invariants of finite groups. 
{\it Research Notes in Mathematics} {\bf 6}. 
A K Peters, Ltd., Wellesley, MA, 1995. 

\bibitem{Solomon}
L. Solomon,
Invariants of finite reflection groups.
{\it Nagoya Math. J.} {\bf 22} (1963), 57--64. 







\bibitem{Steinberg}
R. Steinberg,
On Dickson's theorem on invariants.
{\it J. Fac. Sci. Univ. Tokyo Sect. IA Math.} {\bf 34} (1987), 699--707.

\bibitem{Wood1}
R.M.W.~Wood, 
Modular representations of $GL(n,\FF_p)$ and homotopy theory. 
{\it Algebraic topology, G\"ottingen 1984}, 188--203,
{\it Lecture Notes in Math.} {\bf 1172} Springer, Berlin, 1985. 

\bibitem{Wood2}
R.M.W.~Wood, 
Problems in the Steenrod algebra. 
{\it Bull. London Math. Soc.} {\bf 30} (1998), 449--517. 


\end{thebibliography}
\end{document}